\documentclass[a4paper]{article}

\usepackage{amsfonts}
\usepackage{amsmath}
\usepackage{amssymb}
\usepackage{amsthm}
\usepackage[UKenglish]{babel}
\usepackage{Baskervaldx}
\usepackage{bm}
\usepackage{booktabs}
\usepackage{cite}
\usepackage[T1]{fontenc}
\usepackage{tabularx}
	\newcolumntype{Y}{>{\centering\arraybackslash}X}
\usepackage{tikz}
	\usetikzlibrary{cd}
	\usetikzlibrary{decorations.markings}
\usepackage{url}
\usepackage{xcolor}
	\definecolor{Blue}{HTML}{3d25b9}
	\definecolor{Pink}{HTML}{ec028d}

\usepackage{enumitem}
	\setlist{topsep=0pt,itemsep=0pt}
\usepackage[hang,flushmargin]{footmisc}
\usepackage[top=25mm,bottom=25mm,left=24.8mm,right=24.8mm]{geometry}
\usepackage{microtype}
\usepackage{sectsty}
	\sectionfont{\large}
	\subsectionfont{\normalsize}
	\subsubsectionfont{\normalsize}
\usepackage{titlesec}
	\titlespacing{\section}{0pt}{12pt}{0pt}
	\titlespacing{\subsection}{0pt}{6pt}{0pt}
	\titlespacing{\subsubsection}{0pt}{6pt}{0pt}
\usepackage{titling}
\usepackage[nottoc]{tocbibind} 
\usepackage{tocloft}

\usepackage[colorlinks=true,linkcolor=Blue,citecolor=Blue]{hyperref}
	\hypersetup{breaklinks=true}
\usepackage[capitalise,noabbrev]{cleveref}
	\crefname{equation}{equation}{equations}
	\crefname{conjecture}{Conjecture}{Conjectures}

\theoremstyle{plain}
	\newtheorem{theorem}{Theorem}
	\newtheorem{proposition}[theorem]{Proposition}
	\newtheorem{corollary}[theorem]{Corollary}
	\newtheorem{lemma}[theorem]{Lemma}
	\newtheorem{conjecture}[theorem]{Conjecture}
	\newtheorem{theoremx}{Theorem}
	
	\newtheorem{conjecturex}[theoremx]{Conjecture}
	
	\numberwithin{theorem}{section}
\theoremstyle{definition}
	\newtheorem{definition}[theorem]{Definition}
	\newtheorem{example}[theorem]{Example}
\theoremstyle{remark}
	\newtheorem{remark}[theorem]{Remark}

\newcommand{\dd}{\mathrm{d}}
\newcommand{\h}{\hbar}
\renewcommand{\leq}{\leqslant}
\renewcommand{\geq}{\geqslant}

\newcommand{\UM}{\mathbf{U}}			
\newcommand{\um}{U}					
\DeclareMathOperator{\UW}{Wg^{\UM}}		

\newcommand{\SM}{\mathbf{S}}			
\newcommand{\sm}{S}					
\DeclareMathOperator{\SW}{Wg^{\SM}}		

\newcommand{\SG}{S} 					

\newcommand{\dhn}[2]{\vec{w}^{t}_{#1}(#2)}
\newcommand{\hn}[2]{\vec{h}^{t}_{#1}(#2)}
\newcommand{\dHN}[3]{\vec{W}^t_{#1,#2}(#3)}
\newcommand{\HN}[3]{\vec{H}^t_{#1,#2}(#3)}

\title{Integration on complex Grassmannians, deformed monotone Hurwitz numbers, and interlacing phenomena}
\author{Xavier Coulter \and Norman Do \and Ellena Moskovsky}

\begin{document}

\makeatletter
\textbf{\large \thetitle}

\textbf{\theauthor}
\makeatother

Department of Mathematics, The University of Auckland, Auckland 1142 New Zealand \\
School of Mathematics, Monash University, VIC 3800 Australia \\
School of Mathematics, Monash University, VIC 3800 Australia \\
Email: \href{mailto:xavier.coulter@auckland.ac.nz}{xavier.coulter@auckland.ac.nz}, \href{mailto:norm.do@monash.edu}{norm.do@monash.edu}, \href{mailto:ellena.moskovsky@gmail.com}{ellena.moskovsky@gmail.com}

{\em Abstract.} We introduce a family of polynomials, which arise in three distinct ways: in the large $N$ expansion of a matrix integral, as a weighted enumeration of factorisations of permutations, and via the topological recursion. More explicitly, we interpret the complex Grassmannian $\mathrm{Gr}(M,N)$ as the space of $N \times N$ idempotent Hermitian matrices of rank $M$ and develop a Weingarten calculus to integrate products of matrix elements over it. In the regime of large $N$ and fixed ratio $\frac{M}{N}$, such integrals have expansions whose coefficients count factorisations of permutations into monotone sequences of transpositions, with each sequence weighted by a monomial in $t = 1 - \frac{N}{M}$. This gives rise to the desired polynomials, which specialise to the monotone Hurwitz numbers when $t = 1$.

These so-called deformed monotone Hurwitz numbers satisfy a cut-and-join recursion, a one-point recursion, and the topological recursion. Furthermore, we conjecture on the basis of overwhelming empirical evidence that the deformed monotone Hurwitz numbers are real-rooted polynomials whose roots satisfy remarkable interlacing phenomena.

An outcome of our work is the viewpoint that the topological recursion can be used to ``topologise'' sequences of polynomials, and we claim that the resulting families of polynomials may possess interesting properties. As a further case study, we consider a weighted enumeration of dessins d'enfant and conjecture that the resulting polynomials are also real-rooted and satisfy analogous interlacing properties.

\emph{Acknowledgements.} X.C. was supported by a University of Auckland Doctoral Scholarship. N.D. was supported by the Australian Research Council grant DP180103891. E.M. was supported by an Australian Government Research Training Program (RTP) Scholarship and a Monash University Postgraduate Publication Award. N.D. would like to thank Maksim Karev for planting the seed of an idea that led to this work by asking the question: ``What does the topological recursion for Schr\"{o}der numbers calculate?''

\emph{2020 Mathematics Subject Classification.} 05A15, 05E10, 15B52, 60B20

~

\hrule

~

\tableofcontents

~

\hrule

~

\section{Introduction} \label{sec:introduction}

In this paper, we introduce a family of polynomials, which arise in three distinct ways: in the large $N$ expansion of a matrix integral, as a weighted enumeration of factorisations of permutations, and via the topological recursion. Our construction simultaneously generalises the Narayana polynomials and the monotone Hurwitz numbers, both of which have garnered considerable attention in the literature~\cite{gou-gua-nov14,sul00}. We prove or conjecture that the family of polynomials we introduce satisfies a number of remarkable properties concerning their coefficients and roots, such as symmetry, unimodality, real-rootedness and interlacing.

This work is inspired by the known relations between Weingarten calculus on unitary groups, Jucys--Murphy elements in the symmetric group algebra, and monotone Hurwitz numbers~\cite{nov10,gou-gua-nov14}. Our starting point is the space of $N \times N$ idempotent Hermitian matrices of rank $M$, where $M < N$. This space admits the following three descriptions, where $I_M$ denotes the $M \times M$ identity matrix and $I_{M,N}$ denotes the $N \times N$ matrix whose first $M$ diagonal entries are 1 and whose remaining entries are 0.
\begin{align*}
\SM(M,N) &= \{ S \in \mathrm{Mat}_{N \times N}(\mathbb{C}) \mid S^2 = S, S = S^* \text{ and } \mathrm{rank}(S) = M \} \\
&= \{ S = U^*U \mid U \in \mathrm{Mat}_{M \times N}(\mathbb{C}) \text{ and } UU^* = I_M \} \\
&= \{ S = U I_{M,N} U^* \mid U \in \UM(N) \}
\end{align*}
The unitary group $\UM(N)$ acts transitively on $\SM(M,N)$ by conjugation, thus endowing it with the structure of a homogeneous space. Since the stabiliser of $I_{M,N}$ is $\UM(M) \times \UM(N-M)$, we may identify $\SM(M,N)$ with the complex Grassmannian $\mathrm{Gr}(M,N) \cong \UM(N) \, / \, \UM(M) \times \UM(N-M)$, which parametrises $M$-dimensional subspaces of an $N$-dimensional complex vector space.

As a compact homogeneous space, $\SM(M,N)$ inherits a normalised $\UM(N)$-invariant Haar measure, which we denote succinctly by $\dd S$. For $1 \leq i, j \leq N$, define the function $S_{ij}: \SM(M,N) \to \mathbb{C}$ corresponding to a matrix element. Taking our cue from the general theory of Weingarten calculus, we consider integrals of the form
\[
\int_{\SM(M,N)} S_{i_1 j_1} S_{i_2 j_2} \cdots S_{i_k j_k} \, \dd S,
\]
where $1 \leq i_1, i_2, \ldots, i_k, j_1, j_2, \ldots, j_k \leq N$. Our primary goal is to study such matrix integrals in the regime of large $N$ and fixed ratio $\frac{M}{N}$.

In recent decades, Weingarten calculus has developed into a rich theory concerned with integration on compact groups and related objects, with respect to the Haar measure~\cite{col-mat-nov22}. Following the usual paradigm, we define a Weingarten function that takes as input a permutation $\sigma \in S_k$ and outputs the following elementary integral, where our notation suppresses the dependence on $M$ and $N$.
\[
\SW(\sigma) = \int_{\SM(M,N)} S_{1,\sigma(1)} S_{2,\sigma(2)} \cdots S_{k,\sigma(k)} \, \dd S
\]
Modern approaches to Weingarten calculus typically rely on abstract algebraic tools such as Schur--Weyl duality~\cite{col03, col-sni06}. These are not as amenable to the current setting as the ideas rooted in the pioneering work of Weingarten~\cite{wei78} and revisited more recently by Collins and Matsumoto~\cite{col-mat17}, which we use as inspiration to obtain the following.\footnote{This introduction includes only a cursory discussion of our main results and conjectures. The reader is encouraged to consult the main text for the relevant definitions, precise statements, and further details.}

\begin{theoremx}[Weingarten calculus] ~
\begin{itemize}
\item {\em Convolution formula} (\cref{thm:convolution}) \\
Arbitrary integrals of monomials in the matrix elements of $\SM(M,N)$ reduce to elementary integrals via the equation
\[
\int_{\SM(M,N)} \sm_{i_1j_1} \sm_{i_2j_2} \cdots \sm_{i_kj_k} \, \dd \sm = \sum_{\sigma \in \SG_k} \delta_{i_{\sigma(1)},j_1} \delta_{i_{\sigma(2)},j_2} \cdots \delta_{i_{\sigma(k)},j_k} \SW(\sigma).
\]

\item {\em Orthogonality relations} (\cref{thm:orthogonality}) \\
For each permutation $\sigma \in S_k$, the Weingarten function satisfies the relation
\[
\SW(\sigma) = -\frac{1}{N} \sum_{i=1}^{k-1} \SW(\sigma \circ (i~k)) + \delta_{\sigma(k),k} \, \frac{M}{N} \SW(\sigma^\downarrow) + \frac{1}{N} \sum_{i=1}^{k-1} \delta_{\sigma(i),k} \SW([\sigma \circ (i~k) ]^\downarrow).
\]
Here, for any permutation $\rho \in \SG_k$ that fixes $k$, we write $\rho^\downarrow \in \SG_{k-1}$ to denote the permutation that agrees with $\rho$ on the set $\{1, 2, \ldots, k-1\}$.
\end{itemize}
\end{theoremx}

The orthogonality relations allow for at least two distinct approaches to calculate values of the Weingarten function. By solving the non-degenerate linear system provided by the orthogonality relations, one can explicitly obtain values of the Weingarten function, which are rational functions of $M$ and $N$. Alternatively, by iteratively and infinitely applying the orthogonality relations, one obtains the following large $N$ expansion for the Weingarten function, for fixed ratio $\frac{M}{N}$.

\begin{theoremx} [Large $N$ expansion, \cref{thm:fixTexpansion}] \label{thm:theoremb}
For a permutation $\sigma \in S_k$, let $\dhn{r}{\sigma}$ denote the weighted enumeration of tuples $\bm{\tau} = (\tau_1, \tau_2, \ldots, \tau_r)$ of transpositions in $S_k$ such that
\begin{itemize}
\item $\tau_1 \tau_2 \cdots \tau_r = \sigma$;
\item the sequence $\bm{\tau}$ is monotone --- that is, if we write $\tau_i = (a_i ~ b_i)$ with $a_i < b_i$, then $b_1 \leq b_2 \leq \cdots \leq b_r$; and
\item the weight of $\bm{\tau}$ is $t^{|\{b_1, b_2, \ldots, b_r\}|}$.
\end{itemize}
The Weingarten function has the following large $N$ expansion, for fixed $t= 1 - \frac{N}{M}$.
\[
\SW(\sigma) = \frac{1}{(1-t)^k} \sum_{r=0}^\infty \dhn{r}{\sigma} \left( \frac{-1}{N} \right)^r
\]
\end{theoremx}

Monotone Hurwitz numbers count monotone sequences of transpositions with prescribed length and with product of a prescribed cycle type, usually with an additional transitivity assumption. They are known to arise in the Weingarten calculus for unitary groups and in the large $N$ expansion of the HCIZ matrix integral~\cite{gou-gua-nov14}. The theorem above motivates a ``deformed'' version of the monotone Hurwitz numbers, in which each sequence $((a_1~b_1), (a_2~b_2), \ldots, (a_r~b_r))$ of transpositions is weighted by the monomial $t^{|\{b_1, b_2, \ldots, b_r\}|}$. The resulting polynomials in $t$ are referred to as {\em deformed monotone Hurwitz numbers} and denoted by $\HN{g}{n}{\mu_1, \ldots, \mu_n}$. Their precise definition appears in \cref{def:deformed} and they are the central objects introduced and studied in the present work.

The deformed monotone Hurwitz numbers obey various recursions, from which known results for the usual monotone Hurwitz numbers are recovered when $t = 1$~\cite{gou-gua-nov13a, do-dye-mat17, cha-do21}.

\begin{theoremx}[Recursions] \label{thmx:recursions} ~
\begin{itemize}
\item {\em Cut-and-join recursion} (\cref{thm:cutjoin}) \\
The deformed monotone Hurwitz numbers can be computed from the base case $\HN{0}{1}{1} = 1$ and the recursion
\begin{align*}
\mu_1 \HN{g}{n}{\mu_1, \mu_S} &= \sum_{i=2}^n (\mu_1 + \mu_i) \, \HN{g}{n-1}{\mu_1 + \mu_i, \mu_{S \setminus \{i\}}} + (t - 1)(\mu_1-1) \, \HN{g}{n}{\mu_1-1, \mu_S} \\
&+ \sum_{\alpha+\beta=\mu_1} \alpha \beta \Bigg[ \HN{g-1}{n+1}{\alpha, \beta, \mu_S} + \sum_{\substack{g_1 + g_2 = g \\ I_1 \sqcup I_2 = S}} \HN{g_1}{|I_1|+1}{\alpha, \mu_{I_1}} \, \HN{g_2}{|I_2|+1}{\beta, \mu_{I_2}} \Bigg].
\end{align*}

\item {\em One-point recursion} (\cref{thm:onepoint}) \\ The one-point deformed monotone Hurwitz numbers --- in other words, those with $n = 1$ --- satisfy
\[
d^2 \, \HN{g}{1}{d} = (d-1) (2d-3) (t+1) \, \HN{g}{1}{d-1} - (d-2) (d-3) (t-1)^2 \, \HN{g}{1}{d-2} + d^2(d-1)^2 \, \HN{g-1}{1}{d}.
\]

\item {\em Topological recursion} (\cref{thm:TR}) \\ The deformed monotone Hurwitz numbers are governed by the topological recursion on the genus zero spectral curve $xy^2 + (t-1)xy - y + 1 = 0$.
\end{itemize}
\end{theoremx}

These recursions are all effective and can be used to produce explicit data, some of which is contained in \cref{app:data}. At the level of coefficients, each deformed monotone Hurwitz number is symmetric (the sequence of coefficients is palindromic) and unimodal (the sequence of coefficients increases to a point and then decreases). We prove this in \cref{prop:symmetry} using the cut-and-join recursion and note that these properties are not immediate from the combinatorial deformation of the deformed monotone Hurwitz numbers.

At the level of roots, it appears that the deformed monotone Hurwitz numbers are real-rooted and that they exhibit interlacing phenomena. Two real-rooted polynomials are said to {\em interlace} if their degrees differ by one and their roots weakly alternate on the real number line. We have gathered overwhelming numerical evidence to support the following conjectures.

\begin{conjecturex}[Roots] ~
\begin{itemize}
\item {\em Real-rootedness} (\cref{con:realrooted}) \\ The deformed monotone Hurwitz number $\HN{g}{n}{\mu_1, \mu_2, \ldots, \mu_n}$ is a real-rooted polynomial in $t$.

\item {\em Interlacing} (\cref{con:interlacing}) \\ The polynomial $\HN{g}{n}{\mu_1, \mu_2, \ldots, \mu_n}$ interlaces each of the $n$ polynomials
\[
\HN{g}{n}{\mu_1+1, \mu_2, \ldots, \mu_n}, \quad \HN{g}{n}{\mu_1, \mu_2+1, \ldots, \mu_n}, \quad \ldots, \quad \HN{g}{n}{\mu_1, \mu_2, \ldots, \mu_n+1}.
\]
\end{itemize}
\end{conjecturex}

In the case $(g,n) = (0,1)$, the deformed monotone Hurwitz numbers recover the sequence of Narayana polynomials via the equation
\begin{equation} \label{eq:hurwitznarayana} 
(\mu+1) \, \HN{0}{1}{\mu+1} = \mathrm{Nar}_{\mu}(t) := \sum_{i=1}^{\mu} \frac{1}{\mu} \binom{\mu}{i} \binom{\mu}{i-1} \, t^i.
\end{equation}
Thus, we consider the deformed monotone Hurwitz numbers to be a ``topological generalisation'' of the Narayana polynomials. We propose the topological recursion as a mechanism to ``topologise'' sequences of polynomials more generally. In particular, we claim that doing so can preserve interesting behaviour in the polynomials, such as symmetry, unimodality, real-rootedness, and interlacing properties. This is not only the case for the deformed monotone Hurwitz numbers, but we also observe these phenomena in the weighted enumeration of dessins d'enfant --- in other words, bicoloured maps --- in which each black vertex in a dessin d'enfant is assigned a multiplicative weight $t$.

Our results concerning integration on complex Grassmannians and deformed monotone Hurwitz numbers suggest various avenues for further research. It would be natural to consider integration on real Grassmannians $\mathrm{Gr}(M,N)$, also in the regime of large $N$ with fixed ratio $\frac{M}{N}$, and the first two authors are currently pursuing this line of investigation. Matsumoto considered Weingarten calculus on compact symmetric spaces~\cite{mat13} and the particular case of the symmetric space AIII bears a strong resemblance to the present work. It would be interesting to further develop the parallels between these two settings. The real-rootedness and interlacing conjectures for deformed monotone Hurwitz numbers (\cref{con:realrooted,con:interlacing}) and the weighted dessin d'enfant enumeration (\cref{con:dessins}) not only require proof, but also invite a deeper exploration of how common these phenomena might be.

The structure of the paper is as follows.
\begin{itemize}
\item In \cref{sec:weingarten}, we develop the Weingarten calculus for integration over $\SM(M,N)$. This includes a convolution formula (\cref{thm:convolution}) and orthogonality relations (\cref{thm:orthogonality}). We use the latter to express the Weingarten function of a permutation as a weighted enumeration of monotone sequences of transpositions (\cref{thm:fixTexpansion}). As a consequence, the Weingarten function can be written succinctly in terms of Jucys--Murphy elements in the symmetric group algebra (\cref{prop:weingartenjucys}).

\item In \cref{sec:monotone}, we define the notion of deformed monotone Hurwitz numbers, which are polynomials in the deformation parameter $t$, motivated by the results of the previous section. We prove ``deformed'' analogues of existing results concerning the usual monotone Hurwitz numbers, such as a character formula (\cref{prop:reptheory}), a cut-and-join recursion (\cref{thm:cutjoin}), and a one-point recursion (\cref{thm:onepoint}). It follows from the cut-and-join recursion that the coefficients of deformed monotone Hurwitz numbers are symmetric and unimodal (\cref{prop:symmetry}). On the basis of extensive numerical evidence, we conjecture that these polynomials are real-rooted (\cref{con:realrooted}) and that their roots satisfy remarkable interlacing phenomena (\cref{con:interlacing}).

\item In \cref{sec:TR}, we briefly introduce the topological recursion of Chekhov, Eynard and Orantin~\cite{che-eyn06, eyn-ora07} and then use a powerful result of Bychkov, Dunin-Barkowski, Kazarian and Shadrin~\cite{BDKS21} to prove that topological recursion on the spectral curve $xy^2 + (t-1)xy - y + 1 = 0$ governs the deformed monotone Hurwitz numbers (\cref{thm:TR}). We propose the topological recursion as a mechanism to produce topological generalisations of sequences of polynomials with interesting properties. As a case study, we consider the weighted enumeration of dessins d'enfant --- in other words, bicoloured maps --- in which each black vertex receives a multiplicative weight $t$. This produces a family of polynomials that also satisfies a cut-and-join recursion (\cref{prop:dessincutjoin}) and the topological recursion (\cref{thm:dessinTR}). In analogy with the case of deformed monotone Hurwitz numbers studied in the previous section, we conjecture that these polynomials exhibit real-rootedness and interlacing phenomena (\cref{con:dessins}).
\end{itemize}

\section{Weingarten calculus} \label{sec:weingarten}

\subsection{Convolution formula and orthogonality relations}

As mentioned in the introduction, the present work is concerned with integration on the complex Grasmannian $\mathrm{Gr}(M,N)$ for $M < N$. We interpret this Grassmannian as the space of $N \times N$ idempotent Hermitian matrices of rank $M$, which admits three equivalent descriptions as per the following definition.

\begin{definition}
Let $I_M$ denote the $M \times M$ identity matrix and $I_{M,N}$ denote the $N \times N$ matrix whose first~$M$ diagonal entries are 1 and whose remaining entries are 0. For $M < N$, define the space
\begin{align*}
\SM(M,N) &= \{ S \in \mathrm{Mat}_{N \times N}(\mathbb{C}) \mid S^2 = S, S = S^* \text{ and } \mathrm{rank}(S) = M \} \\
&= \{ S = U^*U \mid U \in \mathrm{Mat}_{M \times N}(\mathbb{C}) \text{ and } UU^* = I_M \} \\
&= \{ S = U I_{M,N} U^* \mid U \in \UM(N) \}.
\end{align*}
\end{definition}

The unitary group $\UM(N)$ acts transitively on $\SM(M,N)$ by conjugation, thus endowing it with the structure of a compact homogeneous space. Thus, the Haar measure on $\UM(N)$ induces a $\UM(N)$-invariant normalised Haar measure on $\SM(M,N)$, which we denote succinctly by $\dd S$. (See~\cite{die-spa14} for an introduction to Haar measures.) The fact that the stabiliser of $I_{M,N} \in \SM(M,N)$ is $\UM(M) \times \UM(N-M)$ allows us to identify $\SM(M,N)$ with the complex Grassmannian $\mathrm{Gr}(M,N) \cong \UM(N) \, / \, \UM(M) \times \UM(N-M)$.

For $1 \leq i, j \leq N$, define the function $S_{ij}: \SM(M,N) \to \mathbb{C}$ corresponding to the $(i,j)$ matrix element. Our primary goal is to calculate integrals of the form
\[
\int_{\SM(M,N)} S_{i_1 j_1} S_{i_2 j_2} \cdots S_{i_k j_k} \, \dd S,
\]
where $1 \leq i_1, i_2, \ldots, i_k, j_1, j_2, \ldots, j_k \leq N$. We impose the technical assumption that $k \leq N$ for future convenience. However, this assumption has little bearing on our work, since we will study these matrix integrals in the regime of large $N$ and fixed ratio $\frac{M}{N}$. The following elementary integrals will be of particular importance.

\begin{definition}[Weingarten function] \label{def:weingarten}
For each permutation $\sigma \in S_k$, define the integral
\[
\SW(\sigma) = \int_{\SM(M,N)} \sm_{1,\sigma(1)} \sm_{2,\sigma(2)} \cdots \sm_{k,\sigma(k)} \, \dd \sm.
\]
We refer to the function $\SW: \SG_0 \sqcup \SG_1 \sqcup \SG_2 \sqcup \cdots \to \mathbb{C}$ as the {\em Weingarten function} for $\SM(M,N)$. Here, we include the symmetric group $S_0$, whose unique element is denoted $(\,)$ and represents the empty permutation. In the notation for the Weingarten function, we suppress the dependence on $M$ and $N$ to avoid clutter.
\end{definition}

The setup described above sits firmly in the realm of Weingarten calculus, which is broadly concerned with the calculation of integrals on compact groups and related objects with respect to the Haar measure~\cite{col-mat-nov22}. Modern accounts of Weingarten calculus often rely on elegant algebraic approaches via Schur--Weyl duality~\cite{col03, col-sni06}. A direct use of such an argument is not immediately available for the case of integration over $\SM(M,N)$. We instead follow the approach via orthogonality relations utilised by Collins and Matsumoto~\cite{col-mat17}, which is in turn inspired by the ideas contained in the seminal paper of Weingarten~\cite{wei78}.

In the remainder of this section, we develop the Weingarten calculus for integration on $\SM(M,N)$ in three parts. First, we prove a convolution formula that reduces general integrals of monomials in the matrix elements to the elementary ones defined in \cref{def:weingarten}. Second, we prove so-called orthogonality relations that completely determine all values of the Weingarten function. Third, we solve the linear system provided by these orthogonality relations, which connects naturally to the representation theory of the symmetric groups, particularly to the Jucys--Murphy elements in the symmetric group algebra.

\begin{theorem}[Convolution formula] \label{thm:convolution}
Arbitrary integrals of monomials in the matrix elements of $\SM(M,N)$ reduce to elementary integrals via the equation
\[
\int_{\SM(M,N)} \sm_{i_1j_1} \sm_{i_2j_2} \cdots \sm_{i_kj_k} \, \dd \sm = \sum_{\sigma \in \SG_k} \delta_{i_{\sigma(1)},j_1} \delta_{i_{\sigma(2)},j_2} \cdots \delta_{i_{\sigma(k)},j_k} \SW(\sigma).
\]
\end{theorem}

\begin{proof}
Write $S \in \SM(M,N)$ in the form $S = UI_{M,N}U^*$ for $U \in \UM(N)$. Then the two sides of the desired equation can be equivalently expressed as integrals over $\UM(N)$.
\begin{align*}
\text{LHS} &= \sum_{m_1, \ldots, m_k = 1}^N \bigg(\prod_{i=1}^k \left[I_{M,N}\right]_{m_im_i} \bigg) \int_{\UM(N)} \um_{i_1m_1} \cdots \um_{i_km_k} \um_{m_1j_1}^* \cdots \um_{m_kj_k}^* \, \dd \um \\
\text{RHS} &= \sum_{\sigma \in \SG_k} \bigg( \prod_{a=1}^k \delta_{i_{\sigma(a)},j_a} \bigg) \sum_{m_1, \ldots, m_k = 1}^N \bigg(\prod_{i=1}^k \left[I_{M,N}\right]_{m_im_i} \bigg) \int_{\UM(N)} \um_{1m_1} \cdots \um_{km_k} \um_{m_1,\sigma(1)}^* \cdots \um_{m_k,\sigma(k)}^* \, \dd \um
\end{align*}
Thus, the result would follow from the equation
\begin{multline*}
\int_{\UM(N)} \um_{i_1m_1} \cdots \um_{i_km_k} \um_{m_1j_1}^* \cdots \um_{m_kj_k}^* \, \dd \um \\
= \sum_{\sigma \in \SG_k} \delta_{i_{\sigma(1)},j_1} \cdots \delta_{i_{\sigma(k)},j_k} \int_{\UM(N)} \um_{1m_1} \cdots \um_{km_k} \um_{m_1,\sigma(1)}^* \cdots \um_{m_k,\sigma(k)}^* \, \dd \um.
\end{multline*}
However, this is a direct consequence of applying the convolution formula for $\UM(N)$~\cite[Corollary~2.4]{col-sni06} to both sides.
\end{proof}

The following orthogonality relations for $\SW$ are inspired by the orthogonality relations obtained by Collins and Matsumoto in other settings for Weingarten calculus~\cite{col-mat17}. The proof crucially relies on the convolution formula of \cref{thm:convolution}. For future reference, observe that we use the notational convention of composing permutations from right to left.

\begin{theorem}[Orthogonality relations] \label{thm:orthogonality}
For each permutation $\sigma \in S_k$, the Weingarten function satisfies the relation
\[
\SW(\sigma) = -\frac{1}{N} \sum_{i=1}^{k-1} \SW(\sigma \circ (i~k)) + \delta_{\sigma(k),k} \, \frac{M}{N} \SW(\sigma^\downarrow) + \frac{1}{N} \sum_{i=1}^{k-1} \delta_{\sigma(i),k} \SW([\sigma \circ (i~k) ]^\downarrow).
\]
Here, for any permutation $\rho \in \SG_k$ that fixes $k$, we write $\rho^\downarrow \in \SG_{k-1}$ to denote the permutation that agrees with $\rho$ on the set $\{1, 2, \ldots, k-1\}$.
\end{theorem}

\begin{proof}
Let $\sigma \in S_k$ and consider the cases $\sigma(k) = k$ and $\sigma(k) \neq k$ separately.

{\em Case 1.} Suppose $\sigma(k) = k$. \\
Consider the integral
\begin{equation*} \label{eq:orthogonality1}
\sum_{i=1}^N \int_{\SM(M,N)} \sm_{1,\sigma(1)} \sm_{2,\sigma(2)} \cdots \sm_{k-1,\sigma(k-1)} \sm_{i,i} \, \dd \sm. \tag{$\ast$}
\end{equation*}
On the one hand, we can use $\sum_{i=1}^N \sm_{ii} = \mathrm{Tr}(\sm) = \mathrm{Tr}(UI_{M,N}U^*) = \mathrm{Tr}(I_{M,N}) = M$ and the definition of $\SW$ to express the integral as $M \SW(\sigma^\downarrow)$. On the other hand, we can apply the convolution formula directly to each summand. For the $i$th summand, where $1 \leq i < k$, the convolution formula yields
\[
\int_{ \SM(M,N) } \sm_{1,\sigma(1)} \cdots S_{i,\sigma(i)} \cdots \sm_{k-1,\sigma(k-1)} \sm_{i,i} \, \dd \sm = \SW(\sigma) + \SW(\sigma \circ (i~k)).
\]
For the $i$th summand, where $k \leq i \leq N$, the convolution formula yields
\[
\int_{ \SM(M,N) } \sm_{1, \sigma(1)} \sm_{2, \sigma(2)} \cdots \sm_{k-1, \sigma(k-1)} \sm_{i,i} \, \dd \sm = \SW(\sigma).
\]
Adding these contributions over $i = 1, 2, \ldots, N$ and equating with the expression we previously obtained for the integral~($\ast$) leads to
\begin{align} \label{eq:weingarten1}
M \SW(\sigma^\downarrow) &= N \SW(\sigma) + \sum_{i=1}^{k-1} \SW(\sigma \circ (i~k)) \notag \\
\Rightarrow \qquad \SW(\sigma) &= - \frac{1}{N} \sum_{i=1}^{k-1} \SW(\sigma \circ (i~k)) + \frac{M}{N} \SW(\sigma^\downarrow).
\end{align}

{\em Case 2.} Suppose $\sigma(k) \neq k$. \\
Let $j = \sigma^{-1}(k)$ and consider the integral
\begin{equation*} \label{eq:orthogonality2}
\sum_{i=1}^{N} \int_{\SM(M,N)} \left( \sm_{1, \sigma(1)} \cdots \sm_{j-1, \sigma(j-1)} \right) \sm_{j,i} \left( \sm_{j+1, \sigma(j+1)} \cdots \sm_{k-1, \sigma(k-1)} \right) \sm_{i, \sigma(k)} \, \dd S. \tag{$\ast\ast$}
\end{equation*}
On the one hand, we have $S^2 = S$ for all $S \in \SM(M,N)$, so it follows that $\sum_{i=1}^{N} \sm_{j,i} \sm_{i, \sigma(k)} = \sm_{j, \sigma(k)}$. Combining this observation with the convolution formula allows us to express the integral as
\[
\int_{\SM(M,N)} \sm_{1, \sigma(1)} \cdots \sm_{j-1, \sigma(j-1)} \sm_{j, \sigma(k)} \sm_{j+1, \sigma(j+1)} \cdots \sm_{k-1, \sigma(k-1)} \, \dd \sm = \SW(\left[\sigma \circ (j~k)\right]^\downarrow).
\]
On the other hand, we can apply the convolution formula directly to each summand, resulting in a calculation analogous to that of Case 1. For the $i$th summand, where $1 \leq i < k$, the convolution formula yields
\[
\int_{\SM(M,N)} \left( \sm_{1, \sigma(1)} \cdots \sm_{j-1, \sigma(j-1)} \right) \sm_{j,i} \left( \sm_{j+1, \sigma(j+1)} \cdots \sm_{k-1, \sigma(k-1)} \right) \sm_{i, \sigma(k)} \, \dd S = \SW(\sigma) + \SW(\sigma \circ (i~k)).
\]
For the $i$th summand, where $k \leq i \leq N$, the convolution formula yields
\[
\int_{\SM(M,N)} \left( \sm_{1, \sigma(1)} \cdots \sm_{j-1, \sigma(j-1)} \right) \sm_{j,i} \left( \sm_{j+1, \sigma(j+1)} \cdots \sm_{k-1, \sigma(k-1)} \right) \sm_{i, \sigma(k)} \, \dd S = \SW(\sigma).
\]
Adding these contributions over $i = 1, 2, \ldots, N$ and equating with the expression we previously obtained for the integral~($\ast\ast$) leads to
\begin{align} \label{eq:weingarten2}
\SW(\left[\sigma \circ (j~k)\right]^\downarrow) &= N \SW(\sigma) + \sum_{i=1}^{k-1} \SW(\sigma \circ (i~k)) \notag \\
\Rightarrow \qquad \SW(\sigma) &= -\frac{1}{N} \sum_{i=1}^{k-1} \SW(\sigma \circ (i~k)) + \frac{1}{N} \SW(\left[\sigma \circ (j~k)\right]^\downarrow).
\end{align}

Finally, the desired result is obtained by writing the two expressions for $\SW(\sigma)$ obtained in \cref{eq:weingarten1,eq:weingarten2} from the two separate cases in one formula, making use of the Kronecker delta notation.
\end{proof}

The orthogonality relations provide a non-degenerate linear system of equations that uniquely determines the Weingarten function. The example below shows that values of the Weingarten function can be computed explicitly and are rational functions of $M$ and $N$.

\begin{example} \label{ex:weingarten}
By the orthogonality relations of \cref{thm:orthogonality} and the conjugacy invariance of $\SW$, we obtain the following equations.
\begin{align*}
\SW((1)(2)(3)) &= -\frac{1}{N} \left[ \SW((1~3)(2)) + \SW((2~3)(1)) \right] + \frac{M}{N} \SW((1)(2)) \\
&= -\frac{2}{N} \SW((1~2)(3)) + \frac{M}{N} \SW((1)(2)) \\
\SW((1~2)(3)) &= -\frac{1}{N} \left[ \SW((1~3~2)) + \SW((1~2~3) \right] + \frac{M}{N} \SW((1~2)) \\
&= -\frac{2}{N} \SW((1~2~3)) + \frac{M}{N}\SW((1~2)) \\
\SW((1~2~3)) &= -\frac{1}{N} \left[ \SW((2~3)(1)) + \SW((1~2)(3) \right] + \frac{1}{N} \SW((1~2)) \\
&= -\frac{2}{N} \SW((1~2)(3)) + \frac{1}{N} \SW((1~2))
\end{align*}

Using the values $\SW((1)(2)) = \frac{M(MN-1)}{N (N^2-1)}$ and $\SW((12)) = \frac{-M(M - N)}{N (N^2-1)}$, one can solve this linear system of equations to obtain the following unique solution.
\begin{align*}
\SW((1)(2)(3)) &= \frac{-2(MN-2)}{N(N^2-4)} \SW((1~2)) + \frac{M}{N} \SW((1)(2)) = \frac{M(M^2N^2 - 2M^2 - 3MN + 4)}{N (N^2-1) (N^2-4)} \\
\SW((1~2)(3)) &= \frac{MN-2}{N^2-4} \SW((1~2)) = \frac{-M (M-N) (MN-2)}{N (N^2-1) (N^2-4)} \\
\SW((1~2~3))) &=\frac{N-2M}{N^2-4}\SW((1~2)) = \frac{M (M-N) (2M-N)}{N (N^2-1) (N^2-4)}
\end{align*}
In this way, one can begin with the base case $\SW((\,)) = 1$ and inductively obtain $\SW(\sigma)$ for $\sigma\in\SG_k$ in terms of $\SW(\sigma')$ for $\sigma'\in\SG_{k-1}$. Further values can be found in \cref{app:data}.
\end{example}

\subsection{Large \texorpdfstring{$N$}{N} expansion}

A priori, computing the values of the Weingarten function $\SW$ via the integral definition is far from straightforward. However, as evidenced by the calculations of \cref{ex:weingarten}, the orthogonality relations of \cref{thm:orthogonality} uniquely determine the Weingarten function and imply that its values are rational functions of $M$ and $N$. We will shortly see that their structure also leads directly to a large $N$ expansion for $\SW(\sigma)$, with coefficients that enumerate factorisations of $\sigma$ into transpositions that satisfy a certain monotonicity condition. This combinatorial structure can be understood in terms of paths in the Weingarten graph, which encode the result of recursively applying the orthogonality relations ad infinitum. The notion of a Weingarten graph was originally introduced by Collins and Matsumoto~\cite{col-mat17}. In the setting of integration over unitary groups, the Weingarten graph $\mathcal{G}^\UM$ has two types of edges, which reflects the fact that the orthogonality relations in that case express the Weingarten function of a permutation in terms of two types of terms. In the setting of integration over $\SM(M,N)$, we have three terms appearing on the right side of the orthogonality relations, motivating the following definition.

\begin{definition}
Define the {\em Weingarten graph} $\mathcal{G}^{\SM}$ to be the infinite directed graph with vertex set $\SG = \bigsqcup_{i=0}^\infty \SG_i$ and edge set $E = E_A \sqcup E_B \sqcup E_C$, where:
\begin{itemize}
\item the set $E_A$ comprises the ``type $A$'' edges, which are of the form 
\[\begin{tikzcd}
	\sigma & & {\sigma \circ (i~k)}
	\arrow[-stealth, very thick, blue!90!black, from=1-1, to=1-3]
\end{tikzcd}\]
for $\sigma \in \SG_k$ and $1 \leq i <k$;
\item the set $E_B$ comprises the ``type $B$'' edges, which are of the form
\[\begin{tikzcd}
	\sigma & & \sigma^\downarrow
	\arrow[-stealth, very thick, red!90!black, dashed, from=1-1, to=1-3]
\end{tikzcd}\]
for $\sigma \in \SG_k$ with $\sigma(k)=k$; and
\item the set $E_C$ comprises the ``type $C$'' edges, which are of the form
\[\begin{tikzcd}
	\sigma & & {[\sigma\circ(i~k)]^\downarrow}
	\arrow[-stealth, very thick, green!80!black, dotted, from=1-1, to=1-3]
\end{tikzcd}\]
for $\sigma\in \SG_k$ such that $\sigma(i) = k$ for some $1 \leq i <k$.
\end{itemize}
\end{definition}

\begin{figure} [ht!]
\centering
\begin{tikzpicture}[
downstyle/.style={very thick, red!90!black, dashed, decoration={markings, mark= at position 0.55 with {\arrow{stealth}}}, postaction={decorate}},
sidestyle/.style={very thick, blue!90!black, decoration={markings, mark= at position 0.35 with {\arrowreversed{stealth}}, mark= at position 0.65 with {\arrow{stealth}}}, postaction={decorate}},
downsidestyle/.style={very thick, green!80!black, dotted, decoration={markings, mark= at position 0.55 with {\arrow{stealth}}}, postaction={decorate}}
]

\def\x{1.25}
\def\y{1.5}

\draw[downstyle] (0,\y) -- (0,0);

\draw[downstyle] (-\x,2*\y) -- (0,\y);
\draw[downsidestyle] (\x,2*\y) -- (0,\y);

\draw[downsidestyle] (-5*\x,3*\y) -- (-\x,2*\y);
\draw[downstyle] (-3*\x,3*\y) -- (-\x,2*\y);
\draw[downsidestyle] (-\x,3*\y) -- (-\x,2*\y);
\draw[downsidestyle] (\x,3*\y) -- (\x,2*\y);
\draw[downstyle] (3*\x,3*\y) -- (\x,2*\y);
\draw[downsidestyle] (5*\x,3*\y) -- (\x,2*\y);

\draw[sidestyle] (-\x,2*\y) -- (\x,2*\y);

\draw[sidestyle] (-5*\x,3*\y) -- (-3*\x,3*\y);
\draw[sidestyle] (-3*\x,3*\y) -- (-\x,3*\y);
\draw[sidestyle] (-\x,3*\y) -- (\x,3*\y);
\draw[sidestyle] (\x,3*\y) -- (3*\x,3*\y);
\draw[sidestyle] (3*\x,3*\y) -- (5*\x,3*\y);
\draw[sidestyle] (-5*\x,3*\y) to[out=20,in=160] (5*\x,3*\y);

\filldraw [fill=white,draw=white] (-5*\x,3*\y) ellipse (0.75 and 0.3);
\filldraw [fill=white,draw=white] (-3*\x,3*\y) ellipse (0.75 and 0.3);
\filldraw [fill=white,draw=white] (-\x,3*\y) ellipse (0.75 and 0.3);
\filldraw [fill=white,draw=white] (\x,3*\y) ellipse (0.75 and 0.3);
\filldraw [fill=white,draw=white] (3*\x,3*\y) ellipse (0.75 and 0.3);
\filldraw [fill=white,draw=white] (5*\x,3*\y) ellipse (0.75 and 0.3);
\filldraw [fill=white,draw=white] (-\x,2*\y) ellipse (0.6 and 0.3);
\filldraw [fill=white,draw=white] (\x,2*\y) ellipse (0.6 and 0.3);
\filldraw [fill=white,draw=white] (0,1*\y) ellipse (0.5 and 0.3);
\filldraw [fill=white,draw=white] (0,0) ellipse (0.35 and 0.3);

\node at (-5*\x,3*\y) {$(13)(2)$};
\node at (-3*\x,3*\y) {$(1)(2)(3)$};
\node at (-\x,3*\y) {$(23)(1)$};
\node at (\x,3*\y) {$(123)$};
\node at (3*\x,3*\y) {$(12)(3)$};
\node at (5*\x,3*\y) {$(132)$};

\node at (-\x,2*\y) {$(1)(2)$};
\node at (\x,2*\y) {$(12)$};

\node at (0,\y) {$(1)$};

\node at (0,0) {$(\,)$};
\end{tikzpicture}
\caption{The part of the Weingarten graph $\mathcal{G}^\SM$ induced by vertices belonging to $\SG_0 \sqcup \SG_1 \sqcup \SG_2 \sqcup \SG_3$.}
\end{figure}
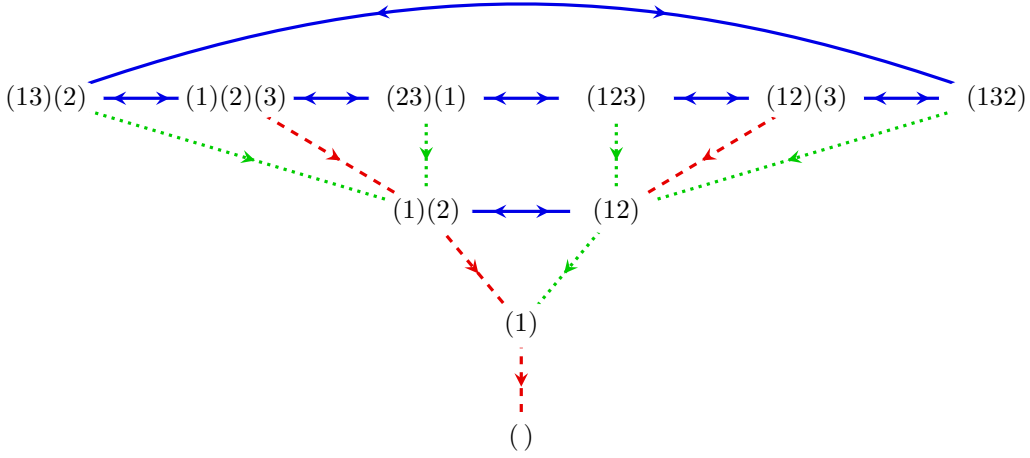

\begin{remark}
The Weingarten graph $\mathcal{G}^\UM$ in the setting of integration over $\UM(N)$ appears in~\cite{col-mat17} and is the subgraph of $\mathcal{G}^\SM$ obtained by removing all type $C$ edges.
\end{remark}

Repeated application of the orthogonality relations leads to paths in the Weingarten graph $\mathcal{G}^\SM$. In this way, one is motivated to enumerate paths in the Weingarten graph in which edges of types $A$, $B$, $C$ receive multiplicative weights $-\frac{1}{N}$, $\frac{M}{N}$, $\frac{1}{N}$, respectively.

\begin{definition}
A {\em path} in $\mathcal{G}^\SM$ is a sequence of permutations 
\[
\bm{\rho} = (\rho_0, \rho_1, \rho_2, \ldots, \rho_\ell) \in \SG^{\ell+1},
\]
where $(\rho_{i-1}, \rho_{i})$ is a directed edge of $\mathcal{G}^\SM$ for each $i = 1, 2, \ldots, \ell$. We call the integer $\ell = \ell(\bm{\rho})$ the {\em length} of the path and denote by $\mathcal{P}(\sigma, \sigma^\prime)$ the set of all paths from $\sigma$ to $\sigma'$. Define the {\em weight} $w(\bm{\rho})$ of a path $\bm{\rho}$ by 
\[
\left( -\frac{1}{N} \right)^{\ell_A(\bm{\rho})} \left( \frac{M}{N} \right)^{\ell_B(\bm{\rho})} \left( \frac{1}{N} \right)^{\ell_C(\bm{\rho})},
\]
where $\ell_K(\bm{\rho})$ denotes the number of edges of type $K\in \{A,B,C\}$ on the path $\bm{\rho}$.
\end{definition}

By construction, the Weingarten graph is a combinatorial encoding of the orthogonality relations and their repeated application. Starting with $\sigma \in \SG_k$, iterating the orthogonality relations $\ell$ times produces the formula
\[
\SW(\sigma) = \sum_{\substack{\bm{\rho} \in \mathcal{P}(\sigma,(\,)) \\ \ell(\bm{\rho}) \leq \ell }} w(\bm{\rho}) + \sum_{\sigma' \in \SG_1 \sqcup \cdots \sqcup \SG_k} \SW(\sigma') \sum_{\substack{\bm{\rho} \in \mathcal{P}(\sigma, \sigma') \\ \ell(\bm{\rho}) = \ell}} w(\bm{\rho}).
\]
By sending the number of iterations $\ell$ to infinity, one obtains the following large $N$ expansion of the Weingarten function.

\begin{proposition}
For any permutation $\sigma$, we have the large $N$ expansion 
\[
\SW(\sigma) = \sum_{\bm{\rho} \in \mathcal{P}(\sigma,(\,))} w(\bm{\rho}) = \sum_{\bm{\rho} \in \mathcal{P}(\sigma,(\,))} \bigg( -\frac{1}{N} \bigg)^{\ell_A(\bm{\rho})} \bigg( \frac{M}{N} \bigg)^{\ell_B(\bm{\rho})} \bigg( \frac{1}{N} \bigg)^{\ell_C(\bm{\rho})}.
\]
\end{proposition}

At this point, let us consider the Weingarten function in the regime of large $N$ with fixed ratio $q = \frac{M}{N}$.\footnote{One could of course consider other regimes, such as fixed $M$, although this line of investigation did not appear to be as fruitful.} For $\bm{\rho} \in \mathcal{P}(\sigma,(\,))$ and $\sigma \in \SG_k$, we have $\ell_B(\bm{\rho}) + \ell_C(\bm{\rho}) = k$, which leads to
\[
\SW(\sigma) = q^k \sum_{r=0}^\infty \left( -\frac{1}{N} \right)^r \sum_{\substack{\bm{\rho} \in \mathcal{P}(\sigma,(\,)) \\ \ell_A(\bm{\rho}) + \ell_C(\bm{\rho}) = r }} \left( -\frac{1}{q} \right)^{\ell_C(\bm{\rho})}.
\]
To develop a clearer combinatorial description of the coefficients appearing in the expansion above, we consider the correspondence between paths in the Weingarten graph and monotone factorisations. This connection already appears in the Weingarten calculus for unitary groups~\cite{col-mat17}.

\begin{definition} \label{def:monotone}
A {\em monotone factorisation} of $\sigma \in \SG_k$ is a sequence $\bm{\tau} = (\tau_1, \tau_2, \ldots, \tau_r)$ of transpositions in $\SG_k$ such that
\begin{itemize}
\item $\tau_1 \tau_2 \cdots \tau_r = \sigma$; and
\item if we write $\tau_i = (a_i ~ b_i)$ with $a_i < b_i$, then $b_1 \leq b_2 \leq \cdots \leq b_r$.
\end{itemize}
Moreover, we call a monotone factorisation {\em transitive} if $\tau_1, \tau_2, \ldots, \tau_r$ generate a transitive subgroup of $\SG_k$. Denote the set of monotone factorisations of $\sigma$ by $\mathcal{M}(\sigma)$ and the length of $\bm{\tau}$ by $r(\bm{\tau})$. We refer to the number of distinct elements of the multiset $\{b_1, b_2, \ldots, b_r\}$ as the {\em hive number} of $\bm{\tau}$ and denote this quantity by $b(\bm{\tau})$.
\end{definition}

Given a path in the Weingarten graph $\mathcal{G}^\SM$ from $\sigma$ to $(\,)$, one can record the sequence of transpositions arising from the type $A$ edges and the type $C$ edges. The composition of these transpositions in reverse order recovers the permutation $\sigma$. Thus, a path $\bm{\rho} \in \mathcal{P}(\sigma,(\,))$ gives rise to a monotone factorisation $\bm{\tau}\in \mathcal{M}(\sigma)$ satisfying $\ell_A(\bm{\rho}) + \ell_C(\bm{\rho}) = r(\bm{\tau})$. Represent this construction via the map
\[
\mathcal{F}: \mathcal{P}(\sigma,(\,)) \to \mathcal{M}(\sigma).
\]

In the analogous construction for the unitary case, one obtains a one-to-one correspondence between paths and monotone factorisations, but that is not the case here. Given a monotone factorisation $\bm{\tau}$ of $\sigma$, there exists a unique path in $\mathcal{F}^{-1}(\bm{\tau})$ containing only edges of types $A$ and $B$. The number of pairs of consecutive $A$--$B$ edges in the path is equal to the hive number $b(\bm{\tau})$. Any such pair can be replaced by a type $C$ edge to obtain an element of $\mathcal{F}^{-1}(\bm{\tau})$ and every element of $\mathcal{F}^{-1}(\bm{\tau})$ can be obtained in this way. Thus, we have $|\mathcal{F}^{-1}(\bm{\tau})| = 2^{b(\bm{\tau})}$. Moreover, keeping track of the effect on edge weights, we have the equality
\[
\sum_{\bm{\rho} \in \mathcal{F}^{-1}(\bm{\tau})} \left(-\frac{1}{q}\right)^{\ell_C(\bm{\rho})} = \left(1 -\frac{1}{q}\right)^{b(\bm{\tau})}.
\]
These observations allow us to express the coefficients of the large $N$ expansion of $\SW(\sigma)$ in terms of monotone factorisations via the equation
\[
\SW(\sigma) = q^k \sum_{r=0}^\infty \left( -\frac{1}{N} \right)^r \sum_{\substack{\bm{\tau} \in \mathcal{M}(\sigma) \\ r(\bm{\tau}) = r }} \left( 1 -\frac{1}{q} \right)^{b(\bm{\tau})}.
\]

The discussion above motivates the following definition and theorem, whose proof is immediate after setting $t = 1 - \frac{1}{q}$ in the previous equation.

\begin{definition} \label{def:weightedcount}
For $r \geq 0$ and $\sigma$ a permutation, let $\dhn{r}{\sigma}$ denote the weighted count of monotone factorisations of~$\sigma$, where the weight of a monotone factorisation $\bm{\tau}$ is $t^{b(\bm{\tau})}$. Let $\hn{r}{\sigma}$ denote the analogous count restricted to transitive monotone factorisations.
\end{definition}

\begin{theorem} [Large $N$ expansion] \label{thm:fixTexpansion}
For a permutation $\sigma \in \SG_k$, we have the following large $N$ expansion for fixed $t = 1 - \frac{N}{M}$.
\begin{equation}
\label{eq:fixTexpansion}
\SW(\sigma) = \frac{1}{(1-t)^k} \sum_{r=0}^\infty \dhn{r}{\sigma} \left( -\frac{1}{N} \right)^r
\end{equation}
\end{theorem}

\subsection{Representation-theoretic interpretation}

The unitary invariance of the Haar measure implies that the Weingarten function $\SW$ is a function of permutations that is constant on conjugacy classes. It is natural to ask whether it admits a natural description in the class algebra of the symmetric group. Such a description of the unitary Weingarten function $\UW$ is well understood and given in terms of the Jucys--Murphy elements in the symmetric group algebra~\cite{mat-nov13,nov10}. We will show that the Weingarten function $\SW$ can be expressed similarly.

The {\em Jucys--Murphy elements} $J_1, J_2, \ldots, J_k$ are elements of the symmetric group algebra defined by
\[
J_i = (1~i) + (2~i) + \cdots + (i-1~i) \in \mathbb{C}[\SG_k],
\]
where we interpret the formula for $i = 1$ as $J_1 = 0$. They were introduced independently by Jucys~\cite{juc74} and Murphy~\cite{mur81} and their seemingly simple definition belies their remarkable properties. For example, they commute with each other and indeed, generate a maximal commutative subalgebra of $\mathbb{C}[\SG_k]$. Any symmetric function of the Jucys--Murphy elements lies in the class algebra $Z\mathbb{C}[\SG_k]$ and the class expansions of such expressions are of significant interest, appearing in various contexts~\cite{fer12}. Furthermore, the Jucys--Murphy elements are essential elements of the Okounkov--Vershik approach to the representation theory of symmetric groups~\cite{oko-ver96}. We will subsequently require the following results that date back to the seminal work of Jucys.

\begin{proposition} [Jucys~\cite{juc74}] \label{prop:jucys} ~
\begin{enumerate}[label=(\alph*)]
\item The Jucys--Murphy elements $J_1, J_2, \ldots, J_k \in \mathbb{C}[\SG_k]$ satisfy
\[
(x + J_1) (x + J_2) \cdots (x + J_k) = \sum_{\sigma \in S_k} x^{\# \mathrm{cycles}(\sigma)} \, \sigma.
\]
\item Let $\lambda$ be a partition of $k$ and let $\chi^\lambda$ the corresponding irreducible character of the symmetric group. We adopt the usual abuse of notation and consider $\chi^\lambda$ as an element of the class algebra $Z\mathbb{C}[\SG_k]$. If $f$ is a symmetric polynomial in $n$ variables, then
\[
f(J_1, J_2, \ldots, J_k) \, \chi^\lambda = f( \mathrm{cont}(\lambda) ) \, \chi^\lambda,
\]
where $\mathrm{cont}(\lambda)$ denotes the multiset of contents in the Young diagram for $\lambda$. (The content of a box in row $i$ and column $j$ of a Young diagram is the number $j-i$.)
\end{enumerate}
\end{proposition}

It was shown by Novak~\cite{nov10} that the unitary Weingarten function can be naturally expressed in terms of the Jucys--Murphy elements via the formula
\[
\sum_{\sigma \in \SG_k} \UW(\sigma) \, \sigma = \prod_{i = 1}^k \dfrac{1}{N+J_i}.
\]
By considering the right side as a series in $-\frac{1}{N}$, one observes that the coefficients are homogeneous symmetric functions of the Jucys--Murphy elements. This leads directly to the notion of monotone Hurwitz numbers, which count monotone factorisations of a given permutation with a prescribed length. The analogue of the equation above for the Weingarten function $\SW$ is the following, which leads to a deformation of the monotone Hurwitz numbers, which are now weighted counts of monotone factorisations with weight equal to $t$ to the power of the hive number. This idea was already briefly introduced in \cref{def:weightedcount}, but will be studied in further detail in \cref{sec:monotone}.

\begin{proposition} \label{prop:weingartenjucys}
For each positive integer $k$, we have the following equality in $Z\mathbb{C}[\SG_k]$, the centre of the symmetric group algebra.
\[
\sum_{\sigma \in \SG_k} \SW(\sigma) \, \sigma = \prod_{i = 1}^k \dfrac{M+J_i}{N+J_i}
\]
\end{proposition}

\begin{proof}
This is essentially an algebraic reformulation of the orthogonality relations of \cref{thm:orthogonality} and we will prove it by induction on $k$. The base case $k=1$ corresponds to the fact that $\SW((1)) = \frac{M}{N}$, which is an immediate consequence of the orthogonality relation for $\sigma = (1)$.

It remains to show that, for $k \geq 2$,
\[
\sum_{\sigma\in \SG_k} \SW(\sigma) \, \sigma =\frac{M+J_k}{N+J_k} \sum_{\sigma\in \SG_{k-1}} \SW(\sigma) \, \sigma^{\uparrow},
\]
where for any permutation $\sigma \in \SG_{k-1}$, we write $\sigma^\uparrow \in \SG_k$ to denote the permutation that agrees with $\sigma$ on the set $\{1, 2, \ldots, k-1\}$ and fixes $k$. Consider multiplying both sides of this equation by $N+J_k$ and use the definition of the Jucys--Murphy element $J_k$ to show that
\begin{align}
(N + J_k) \sum_{\sigma \in \SG_k} \SW(\sigma) \, \sigma &= \sum_{\sigma \in \SG_k} \bigg( N\SW(\sigma) + \sum_{i=1}^{k-1} \SW((i ~ k) \circ \sigma) \bigg) \sigma, \label{eq:proof1} \\
(M+J_k) \sum_{\sigma \in \SG_{k-1}} \SW(\sigma) \, \sigma^{\uparrow} &= \sum_{\sigma \in \SG_k}\bigg( \delta_{\sigma(k),k} \, M \SW(\sigma^\downarrow) + \sum_{i=1}^{k-1} \delta_{\sigma(i),k} \SW( [\sigma \circ (i ~ k)]^\downarrow \bigg) \sigma. \label{eq:proof2}
\end{align}
The second equality is more subtle than the first and relies on the observation that for $1 \leq i \leq k-1$,
\[
\sum_{\sigma\in \SG_{k-1}} \SW(\sigma) \, (i ~ k) \circ \sigma^{\uparrow} = \sum_{\sigma \in \SG_k : \sigma(k) = k} \SW(\sigma^\downarrow) \, (i ~ k) \circ \sigma = \sum_{\sigma \in \SG_k : \sigma(i) = k} \SW([ \sigma\circ (i ~ k)]^\downarrow) \, \sigma.
\]
Finally, the desired equality between \cref{eq:proof1,eq:proof2} follows directly from the orthogonality relations of \cref{thm:orthogonality}, which completes the proof.
\end{proof}

\cref{prop:weingartenjucys} implies the following character formula for the Weingarten function.

\begin{corollary}[Character formula] \label{cor:characterformula}
Consider a permutation $\sigma\in S_k$ and let $\mathrm{id} \in S_k$ be the identity. Then
\[
\SW(\sigma) = \frac{1}{k!} \sum_{\lambda \vdash k} \chi^\lambda(\mathrm{id}) \, \chi^\lambda(\sigma) \prod_{\square \in \lambda} \frac{M+c(\square)}{N+c(\square)},
\]
where $c(\square)$ denotes the content of the box $\square$ in a Young diagram of a partition. Here and throughout, we use the notation $\lambda \vdash k$ to denote that $\lambda$ is a partition of $k$.
\end{corollary}

\begin{proof}
By the orthogonality of characters, the identity in $Z\mathbb{C}[S_k]$ can be expressed as $e = \frac{1}{k!} \sum_{\lambda \vdash k} \chi^\lambda(\mathrm{id}) \, \chi^\lambda$. Use part (b) of \cref{prop:jucys} to write
\[
\prod_{i=1}^k \frac{M+J_i}{N+J_i} \cdot e = \frac{1}{k!} \sum_{\lambda\vdash k} \left( \prod_{\square \in \lambda} \dfrac{M + c(\square)}{N+c(\square)} \right) \chi^{\lambda}(\mathrm{id}) \, \chi^\lambda .
\]
By \cref{prop:weingartenjucys}, $\SW(\sigma)$ is the coefficient of $\sigma$ in the above expression, and the result follows.
\end{proof}

In summary, integrals on the space $\SM(M,N)$ of matrices admit a convolution formula involving the Weingarten function $\SW$. The Weingarten function is in turn determined by orthogonality relations, from which it follows that its values have a representation-theoretic interpretation in terms of Jucys--Murphy elements. An alternative perspective shows that the large $N$ expansion of the Weingarten function for fixed $t = 1 - \frac{N}{M}$ has coefficients that are weighted counts of monotone factorisations. In the next section, we investigate the combinatorics of these enumerations in further detail.

\section{Deformed monotone Hurwitz numbers} \label{sec:monotone}

\subsection{Deformed monotone Hurwitz numbers}

Monotone Hurwitz numbers enumerate monotone factorisations of permutations and arise as coefficients in the large $N$ expansion of the unitary Weingarten function~\cite{gou-gua-nov14}. Weingarten calculus on the space $\SM(M,N)$ naturally motivates a deformation of the monotone Hurwitz numbers, obtained by counting with a weight equal to $t$ to the power of the hive number. This construction produces a family of polynomials with remarkable properties.

Recall from \cref{def:weightedcount} that $\dhn{r}{\sigma}$ denotes the weighted count of monotone factorisations of $\sigma$, while $\hn{r}{\sigma}$ restricts the count to transitive monotone factorisations. To fit the usual nomenclature concerning Hurwitz-type problems, we introduce a new notation to reindex by ``topology'', consider the enumeration as a function on cycle type, and normalise by the product of cycle lengths.

\begin{definition} [Deformed monotone Hurwitz numbers] \label{def:deformed}
For $\mu_1, \ldots, \mu_n \geq 1$, let $\sigma$ be any permutation with $n$ cycles of lengths $\mu_1, \ldots, \mu_n$ and define
\begin{align*}
\dHN{g}{n}{\mu_1, \ldots, \mu_n} &= \frac{1}{\mu_1 \cdots \mu_n} \, \dhn{|\mu|+2g-2+n}{\sigma}, \\
 \HN{g}{n}{\mu_1, \ldots, \mu_n} &= \frac{1}{\mu_1 \cdots \mu_n} \, \hn{|\mu|+2g-2+n}{\sigma}.
\end{align*}
Here and throughout, we use the notation $|\mu|$ as a shorthand for the sum $\mu_1 + \cdots + \mu_n$. We refer to $\dHN{g}{n}{\mu_1, \ldots, \mu_n}$ as a {\em disconnected deformed monotone Hurwitz number} and $ \HN{g}{n}{\mu_1, \ldots, \mu_n}$ as a {\em connected deformed monotone Hurwitz number}.
\end{definition}

\begin{remark}
The pair $(g,n)$ here can be interpreted as encoding the topology of a surface, with $g$ denoting the genus and $n$ the number of punctures, or marked points. A monotone factorisation $(\tau_1, \tau_2, \ldots, \tau_r)$ of $\sigma$ can be interpreted as the monodromy of a branched cover $\mathcal{S} \to \mathbb{CP}^1$ of Riemann surfaces, with $\sigma$ representing the monodromy over $\infty \in \mathbb{CP}^1$. So $\dHN{g}{n}{\mu_1, \ldots, \mu_n}$ and 
$\HN{g}{n}{\mu_1, \ldots, \mu_n}$ enumerate certain branched covers of $\mathbb{CP}^1$ by a genus $g$ surface with $n$ preimages of $\infty \in \mathbb{CP}^1$, the latter restricting the enumeration to connected covers. Consistent with this interpretation, we have that $\HN{g}{n}{\mu_1, \ldots, \mu_n} = 0$ unless $g$ is a non-negative integer, while $\dHN{g}{n}{\mu_1, \ldots, \mu_n}$ can be non-zero when $g$ is a negative integer. Note that the genus of a disconnected surface may be negative, since we consider the Euler characteristic $2 - 2g$ to be additive over disjoint union, rather than the genus itself.
\end{remark}

The character formula of \cref{cor:characterformula} translates into one for deformed monotone Hurwitz numbers as follows.

\begin{proposition} [Character formula] \label{prop:reptheory}
Let $[\h^r] F(\h)$ denote the coefficient of $\h^r$ in the series for $F(\h)$ at $\h = 0$ and let $\chi^\lambda_\mu$ denote the value of the symmetric group character $\chi^\lambda$ on a permutation of cycle type $\mu$. Then
\[
\dHN{g}{n}{\mu_1, \ldots, \mu_n} = \frac{1}{\mu_1 \cdots \mu_n} \, \big[ \h^{|\mu|+2g-2+n} \big] \sum_{\lambda \vdash |\mu|} \frac{\chi^\lambda_{11 \cdots 1} \, \chi^\lambda_\mu}{|\mu|!} \prod_{\square \in \lambda} \frac{1 - \h \, c(\square) + t \h \, c(\square)}{1 - \h \, c(\square)}.
\]
\end{proposition}

\begin{proof}
Equate the expressions for $\SW(\sigma)$ in terms of monotone factorisations (\cref{thm:fixTexpansion}) and characters of the symmetric group (\cref{cor:characterformula}), using the notation $\h = -\frac{1}{N}$ and $t = 1 - \frac{N}{M}$.
\[
\frac{1}{(1-t)^k} \sum_{r=0}^\infty \dhn{r}{\sigma} \, \h^r = \frac{1}{k!} \sum_{\substack{\lambda\vdash k}} \chi^\lambda(\mathrm{id}) \, \chi^\lambda(\sigma) \prod_{\square \in \lambda} \frac{1 - \h \, c(\square) + t \h \, c(\square)}{(1-t) (1 - \h \, c(\square))}
\]
The desired result is obtained by multiplying through by $(1-t)^k$, extracting the coefficient of $\h^{|\mu|+2g-2+n}$ from both sides, and using \cref{def:deformed} for $\dHN{g}{n}{\mu_1, \ldots, \mu_n}$.
\end{proof}

The family of deformed monotone Hurwitz numbers is stored in the large $N$ expansion of the following matrix integral, which we interpret as the partition function for the enumeration. As usual, the partition function naturally stores the disconnected enumeration, while its logarithm stores the connected enumeration.

\begin{proposition} [Matrix integral] \label{prop:integral}
The large $N$ expansion of the formal matrix integral below is a partition function for the deformed monotone Hurwitz numbers. Here, we use the notation $\h = -\frac{1}{N}$ and $p_i$ represents the $i$th power sum symmetric function, evaluated on the eigenvalues of the $N \times N$ complex matrix $A$. (It is common to see a power of $\h^{2g-2+n}$ in the partition function --- the extra $\h^{|\mu|}$ here can be removed by rescaling.)
\begin{align*}
\int_{\SM(M,N)} \exp \frac{N \, \mathrm{tr}(A\sm)}{M} \, \dd \sm &= 1 + \sum_{g=-\infty}^{\infty} \sum_{n=1}^\infty \sum_{\mu_1, \ldots, \mu_n = 1}^\infty \dHN{g}{n}{\mu_1, \ldots, \mu_n} \, \frac{\h^{|\mu|+2g-2+n}}{n!} \, p_{\mu_1} \cdots p_{\mu_n} \\
&= \exp \Bigg[ \sum_{g=0}^{\infty} \sum_{n=1}^\infty \sum_{\mu_1, \ldots, \mu_n = 1}^\infty \HN{g}{n}{\mu_1, \ldots, \mu_n} \, \frac{\h^{|\mu|+2g-2+n}}{n!} \, p_{\mu_1} \cdots p_{\mu_n} \Bigg]
\end{align*}
\end{proposition}

\begin{remark}
Recall that the parameter $t$ is encoded in the parameters of the space $\SM(M,N)$ via the relation $t = 1 - \frac{N}{M}$. The usual monotone Hurwitz numbers are recovered from their deformed counterparts defined above by setting $t = 1$. Perhaps surprisingly, this particular value of $t$ does not correspond to valid parameters and $M$ and $N$ in the matrix integral interpretation. It would be interesting to have a more direct connection between integration over $\UM(N)$ and $\SM(M,N)$.
\end{remark}

\begin{remark}
It would also be natural to consider ``double'' deformed monotone Hurwitz numbers that take two partitions as arguments, rather than one. These enumerate monotone factorisations whose product with a permutation of cycle type given by the first partition produces a permutation of cycle type given by the second. Although this ``double'' enumeration deserves further attention, we refrain from developing this theory in the present work, since they do not give rise to polynomials with the same remarkable properties as the ``single'' enumeration.
\end{remark}

\subsection{Cut-and-join recursion}

The notion of cut-and-join analysis was originally developed for single Hurwitz numbers, before being adapted to work in the monotone case~\cite{gou-jac-vai00, gou-gua-nov13a}. Deformed monotone Hurwitz numbers are amenable to a similar analysis, resulting in the following recursion.

\begin{theorem}[Cut-and-join recursion] \label{thm:cutjoin}
Apart from the initial condition $\HN{0}{1}{1} = 1$, the deformed monotone Hurwitz numbers satisfy
\begin{align} \label{eq:cutandjoin}
\mu_1 \HN{g}{n}{\mu_1, \mu_S} &= \sum_{i=2}^n (\mu_1 + \mu_i) \HN{g}{n-1}{\mu_1 + \mu_i, \mu_{S \setminus\{i\}}} + (t - 1) (\mu_1-1) \HN{g}{n}{\mu_1-1, \mu_S} \nonumber \\
&+ \sum_{\alpha+\beta=\mu_1} \alpha \beta \Bigg[ \HN{g-1}{n+1}{\alpha, \beta, \mu_S} + \sum_{\substack{g_1+g_2 = g \\ I_1 \sqcup I_2= S}} \HN{g_1}{|I_1|+1}{\alpha, \mu_{I_1}} \, \HN{g_2}{|I_2|+1}{\beta, \mu_{I_2}} \Bigg],
\end{align}
where we use the notation $S = \{2, 3, \ldots, n\}$ and $\mu_I = \{\mu_{i_1}, \mu_{i_2}, \ldots, \mu_{i_k}\}$ for $I = \{i_1, i_2, \ldots, i_k\}$.
\end{theorem}

\begin{proof} 
The proof is based on the case for the usual monotone Hurwitz numbers~\cite{gou-gua-nov13a}, with some additional care required to keep track of the deformation parameter $t$. Consider multiplying \cref{eq:cutandjoin} by $\mu_2 \cdots \mu_n$ and consider the right side to be composed of four terms in the obvious way.

Fix a permutation $\sigma \in \SG_k$ of cycle type $(\mu_1, \ldots, \mu_n)$ and consider the quantity $\hn{r+1}{\sigma}$ for some $r \geq 0$. Note that a transitive monotone factorisation $\bm{\tau} = (\tau_1, \tau_2, \ldots, \tau_{r+1})$ of $\sigma$ must have $\tau_{r+1} = (a~k)$ for some $1\leq a \leq k-1$. So a transitive monotone factorisation $\bm{\tau}$ of $\sigma$ with length $r+1$ is equivalent to a monotone factorisation $\bm{\tau}'$ of $\sigma\circ (a~k)$ with length $r$ for some $1 \leq a \leq k-1$, where $\langle \bm{\tau}', (a~k)\rangle \leq \SG_k$ is transitive. This can be expressed via the equation
\[
\hn{r+1}{\sigma} = \sum_{\substack{\bm{\tau} \in \mathcal{M}_{r+1}(\sigma) \\ \langle\bm{\tau}\rangle \text{ transitive}}} t^{b(\bm{\tau})} = \sum_{a=1}^{k-1}\sum_{\substack{\bm{\tau}' \in \mathcal{M}_r(\sigma \circ (a~ k)) \\ \langle\bm{\tau}',(a~k)\rangle \text{ transitive}}} t^{|\bm{b}(\bm{\tau}')\cup\{k\}|},
\]
where we introduce the boldface notation $\bm{b}(\bm{\tau})$ to denote the set $\{b_1, b_2, \ldots, b_r\}$, excluding multiplicities, for $\bm{\tau} = ((a_1~b_1), (a_2~b_2), \ldots, (a_r~b_r))$.

The sum on the right side can be broken down into further sums, depending on various cases. For a fixed value of $a$, either $a$ is in the same cycle of $\sigma$ as $k$, or $a$ is in a different cycle. The transposition $(a~k)$ is referred to as a cut or a join of $\sigma$ depending on these two cases, respectively. Moreover, if $(a~k)$ is a join of $\sigma$, then any monotone factorisation $\bm{\tau}'$ of $\sigma \circ (a~k)$, with $\langle \bm{\tau}', (a~k)\rangle$ transitive, must be transitive itself. Based on these considerations, each term in the equation above falls into one of three cases: for each $1\leq a\leq k-1$ and $\bm{\tau}'\in\mathcal{M}_r(\sigma \circ (a~k))$ with $\langle\bm{\tau}',(a~k)\rangle$ transitive, either
\begin{itemize}
\item $(a~k)$ is a \textit{join} of $\sigma$ and $\langle\bm{\tau}'\rangle \leq \SG_k$ is transitive;
\item $(a~k)$ is a \textit{cut} of $\sigma$ and $\langle\bm{\tau}'\rangle \leq \SG_k$ is transitive; or
\item $(a~k)$ is a \textit{cut} of $\sigma$ and $\langle\bm{\tau}' \rangle \leq \SG_k$ is {\em not} transitive.
\end{itemize}

The first two cases can be expressed as one term to give the equation
\begin{equation} \label{eq:cutandjointemp}
\hn{r+1}{\sigma} = \sum_{a=1}^{k-1} \hn{r}{\sigma \circ (a~k)} + 
 \sum_{\substack{a=1\\ (a~k) \text{ cuts } \sigma}}^{k-1}\sum_{\substack{\bm{\tau}' \in \mathcal{M}_r(\sigma \circ (a~ k)) \\ \langle\bm{\tau}',(a~k)\rangle \text{ transitive}\\\langle\bm{\tau}'\rangle \text{ not transitive} }} t^{|\bm{b}(\bm{\tau}')\cup\{k\}|}.
\end{equation}
In the usual notation for deformed monotone Hurwitz numbers, the first term here precisely gives rise to the first and third terms of \cref{eq:cutandjoin}, while the second term here {\em almost} gives rise to the fourth term of \cref{eq:cutandjoin}. The remainder of the proof explains the meaning of ``almost'' here and how the second term of \cref{eq:cutandjoin} accounts correctly for this anomaly.

If $\bm{\tau}'$ falls into the third case for some value of $1 \leq a \leq k-1$, then the transitive orbit of $\langle \bm{\tau}', (a~k) \rangle$ is split into two orbits of $\langle \bm{\tau}' \rangle$ --- one containing $a$ and one containing $k$. In fact, these orbits are given by a partition of the disjoint cycles of $\sigma \circ (a~k)$ into two classes, thus giving rise to the quadratic terms in \cref{eq:cutandjoin}. The only instance that $b(\bm{\tau}')$ and $b(\bm{\tau}', (a~k))$ differ is when $k$ is in an orbit of size $1$ under $\langle \bm{\tau}' \rangle$, which then requires an additional weight of $t$ to be contributed. The $t-1$ factor in the second term of \cref{eq:cutandjoin} precisely handles this adjustment, since we wish to remove one of the $\HN{0}{1}{1} \, \HN{g}{n}{\mu_1 -1, \vec{\mu}_S}$ contributions appearing in the fourth term and weight it by an additional factor of $t$.

Thus, in the correct final accounting of all terms, we see that the first term of \cref{eq:cutandjointemp} corresponds to the first and third terms of \cref{eq:cutandjoin}, while the second term of \cref{eq:cutandjointemp} corresponds to the second and fourth terms of \cref{eq:cutandjoin}.
\end{proof}

The cut-and-join recursion provides an effective algorithm for calculating deformed monotone Hurwitz numbers, which is more efficient than naive approaches. With this improved computational power comes the ability to make large-scale empirical observations on the structure of the deformed monotone Hurwitz numbers and a recursive means to prove them. The most evident of these observations are the symmetric and unimodal nature of the coefficients of $\HN{g}{n}{\mu_1, \ldots, \mu_n}$.

Adopting the terminology of Brenti~\cite{bre90} and Zeilberger~\cite{zei89}, we define the following.

\begin{definition}
A polynomial $P(t) = a_0 + a_1t + a_2t^2 + \cdots + a_nt^n$ with real coefficients is said to be:
\begin{itemize}
\item {\em symmetric} if the sequence $\ldots, 0, 0, a_0, a_1, a_2, \ldots, a_n, 0, 0, \ldots$ is palindromic;
\item {\em unimodal} if there exists an integer $k$ such that $a_0 \leq a_1 \leq \cdots \leq a_k \geq a_{k+1} \geq \cdots \geq a_n$; and
\item a {\em $\Lambda$-polynomial} if it is a non-zero polynomial with non-negative coefficients that is both palindromic and unimodal.
\end{itemize}
If $P$ is a $\Lambda$-polynomial, then there exists a unique integer $\mathrm{dar}(P)$ such that $P(t) = t^{\mathrm{dar}(P)} P(t^{-1})$. The quantity $\mathrm{dar}(P)$ is called the {\em darga} of $P$.
\end{definition}

\begin{proposition} \label{prop:symmetry}
For $g \geq 0$, $n \geq 1$ and $|\mu| \geq 2$, the deformed monotone Hurwitz number $\HN{g}{n}{\mu_1, \ldots, \mu_n}$ is a $\Lambda$-polynomial of degree $|\mu|-1$ and darga $|\mu|$.
\end{proposition}

\begin{proof} 
We use strong induction on the value of $r = |\mu| + 2g - 2 + n \geq 1$. The only deformed monotone Hurwitz number corresponding to $r = 1$ is $\HN{0}{1}{2} = \frac{1}{2} t$, which is indeed a $\Lambda$-polynomial of degree $1$ and darga $2$. Now consider a deformed monotone Hurwitz number $\HN{g}{n}{\mu_1, \ldots, \mu_n}$ corresponding to $r \geq 2$ and rewrite the cut-and-join recursion of \cref{thm:cutjoin} as
\begin{align*}
\mu_1 \HN{g}{n}{\mu_1, \mu_S} &= \sum_{i=2}^n (\mu_1 + \mu_i) \HN{g}{n-1}{\mu_1 + \mu_i, \mu_{S \setminus\{i\}}} + (t + 1) (\mu_1-1) \HN{g}{n}{\mu_1-1, \mu_S} \nonumber \\
&+ \sum_{\alpha+\beta=\mu_1} \alpha \beta \Bigg[ \HN{g-1}{n+1}{\alpha, \beta, \mu_S} + \sum_{\substack{g_1+g_2 = g \\ I_1 \sqcup I_2= S}}^{\text{no } \HN{0}{1}{1}} \HN{g_1}{|I_1|+1}{\alpha, \mu_{I_1}} \, \HN{g_2}{|I_2|+1}{\beta, \mu_{I_2}} \Bigg].
\end{align*}
Here, we have simply removed the two terms including $\HN{0}{1}{1}$ from the final summation and incorporated them into the second term on the right side.

By the inductive hypothesis, all deformed monotone Hurwitz numbers appearing on the right side of the cut-and-join recursion are $\Lambda$-polynomials. To prove that $\HN{g}{n}{\mu_1, \ldots, \mu_n}$ is also one, we rely on the following closure properties of $\Lambda$-polynomials, proofs of which can be found in~\cite{sun-wan-zha15}.
\begin{itemize}
\item If $P$ and $Q$ are $\Lambda$-polynomials, then the product $PQ$ is a $\Lambda$-polynomial with $\mathrm{dar}(PQ) = \mathrm{dar}(P) + \mathrm{dar}(Q)$.
\item If $P$ and $Q$ are $\Lambda$-polynomials of the same darga $d$, then $P + Q$ is a $\Lambda$-polynomial of darga $d$.
\end{itemize}
It follows that each of the four terms on the right side of the equation above are $\Lambda$-polynomials of darga~$|\mu|$, or possibly equal to the zero polynomial. So the entire right side, and hence $\HN{g}{n}{\mu_1, \ldots, \mu_n}$, is a $\Lambda$-polynomial of darga $|\mu|$.

For any permutation $\sigma \in S_k$ for $k \geq 2$, one can construct a transitive monotone factorisation of $\sigma$ with hive number $1$. So the polynomial $\HN{g}{n}{\mu_1, \ldots, \mu_n}$ has a non-zero linear term, but no constant term. Combined with the fact that it is a $\Lambda$-polynomial of darga $|\mu|$, it follows that $\HN{g}{n}{\mu_1, \ldots, \mu_n}$ is a polynomial of degree $|\mu|-1$.
\end{proof}

\begin{remark}
We have thus far provided two interpretations for the polynomials that we refer to as deformed monotone Hurwitz numbers --- one as weighted counts of monotone factorisations (\cref{def:deformed}) and the other as coefficients in the large $N$ expansion of a matrix integral (\cref{prop:integral}). The symmetry of these polynomials does not appear to be immediately evident from either of these manifestations. In the matrix integral interpretation, the transformation $t \mapsto \frac{1}{t}$ encodes the transformation $M \mapsto N-M$. This is natural from the perspective of integration on the Grassmannian, since $\mathrm{Gr}(M,N) \cong \mathrm{Gr}(N-M,N)$, although further investigation of this symmetry is warranted. In the monotone factorisation interpretation, the symmetry of the polynomials implies that the number of transitive monotone factorisations of a permutation $\sigma \in S_k$ with hive number $h$ is equal to the number of transitive monotone factorisations of $\sigma$ with hive number $k-h$. This does not appear to be immediately obvious, so it would be interesting to have a direct combinatorial proof of this fact. It is worth remarking that the disconnected enumeration $\dHN{g}{n}{\mu_1, \ldots, \mu_n}$ does not lead to symmetric polynomials, so the transitivity condition is important here.
\end{remark}

\begin{remark}
There are various other properties of polynomial coefficients that are often mentioned alongside symmetry and unimodality. The deformed monotone Hurwitz numbers constitute a family of polynomials that appears to satisfy many of these. As examples, we pose the following questions. Do the coefficients of the deformed monotone Hurwitz numbers exhibit log-concavity? Do they exhibit asymptotic normality?
\end{remark}

As we will see in \cref{sec:TR}, deformed monotone Hurwitz numbers belong to a large class of enumerative problems governed by the topological recursion. In many such instances, the one-point invariants --- in other words, those with $n = 1$ --- are governed by a recursion that is linear, rather than quadratic like the cut-and-join recursion. In previous work of Chaudhuri and the second author~\cite{cha-do21}, it was shown that such one-point recursions exist for ``weighted Hurwitz numbers'', to use the terminology of Alexandrov, Chapuy, Eynard and Harnad~\cite{ACEH20}. Furthermore, they gave an explicit algorithmic approach to obtain these recursions that is effective in many examples. In the case of deformed monotone Hurwitz numbers, one obtains the following result.

\begin{theorem}[One-point recursion] \label{thm:onepoint}
The one-point deformed monotone Hurwitz numbers --- in other words, those with $n = 1$ --- satisfy
\[
d^2 \, \HN{g}{1}{d} = (d-1) (2d-3) (t+1) \, \HN{g}{1}{d-1} - (d-2) (d-3) (t-1)^2 \, \HN{g}{1}{d-2} + d^2(d-1)^2 \, \HN{g-1}{1}{d}.
\]
\end{theorem}

Observe that setting $t = 1$ above recovers the known one-point recursion for monotone Hurwitz numbers~\cite{cha-do21}. Furthermore, setting $g = 0$ and combining with \cref{eq:hurwitznarayana} recovers the known linear recursion for Narayana polynomials~\cite{fis06}.

\subsection{Interlacing phenomena} \label{subsec:interlacing}

\subsubsection*{Root conjectures}

The cut-and-join recursion of \cref{thm:cutjoin} can be used to effectively compute a large number of deformed monotone Hurwitz numbers, some of which are shown in \cref{app:data}. We previously stated in \cref{prop:symmetry} that the coefficients of these polynomials are symmetric and unimodal. We now turn our attention to their roots, which exhibit rather striking behaviour that remains largely conjectural at present. The most apparent property concerning roots of deformed monotone Hurwitz numbers is the following.

\begin{conjecture}[Real-rootedness] \label{con:realrooted}
For all $g \geq 0$, $n \geq 1$ and $\mu_1, \ldots, \mu_n \geq 1$, the deformed monotone Hurwitz number $\HN{g}{n}{\mu_1, \mu_2, \ldots, \mu_n}$ is a real-rooted polynomial in $t$.
\end{conjecture}

The roots of the deformed monotone Hurwitz numbers are not only real, but also possess interesting structure relative to each other. We say that a polynomial $P$ {\em interlaces} a polynomial $Q$ if
\begin{itemize}
\item the degree of $P$ is $n$ and the degree of $Q$ is $n+1$ for some positive integer $n$;
\item $P$ has $n$ real roots $a_1 \leq a_2 \leq \cdots \leq a_n$ and $Q$ has $n+1$ real roots $b_1 \leq b_2 \leq \cdots \leq b_{n+1}$, allowing for multiplicity; and
\item $b_1 \leq a_1 \leq b_2 \leq a_2 \leq \cdots \leq b_n \leq a_n \leq b_{n+1}$.
\end{itemize}
If the inequalities above are strict, then we say that the polynomial $P$ {\em strictly interlaces} the polynomial $Q$. By convention, we will also say that a polynomial $P$ (strictly) interlaces a polynomial $Q$ if $P$ is constant and $Q$ is affine.

It is known that the sequence of Narayana polynomials interlace in the sense that $\mathrm{Nar}_\mu(t)$ interlaces $\mathrm{Nar}_{\mu+1}(t)$ for every positive integer $\mu$~\cite{fis06}. Given \cref{eq:hurwitznarayana}, one may wonder whether an analogous property persists more generally across the whole family of deformed monotone Hurwitz numbers. Overwhelming empirical evidence suggests that this is indeed the case and we have the following.

\begin{conjecture}[Interlacing] \label{con:interlacing}
The polynomial $\HN{g}{n}{\mu_1, \mu_2, \ldots, \mu_n}$ interlaces each of the $n$ polynomials
\[
\HN{g}{n}{\mu_1+1, \mu_2, \ldots, \mu_n}, \quad \HN{g}{n}{\mu_1, \mu_2+1, \ldots, \mu_n}, \quad \ldots, \quad \HN{g}{n}{\mu_1, \mu_2, \ldots, \mu_n+1}.
\]
\end{conjecture}

\cref{con:interlacing} states that for fixed $(g,n)$, one obtains a lattice of interlacing polynomials. It has been checked computationally for many cases, including the following, amounting to 1430 independent checks:
\begin{itemize}
\item $g \in \{0, 1\}$, $n \in \{1, 2, 3, 4, 5\}$, $|\mu| \leq 15$;
\item $g \in \{2, 3\}$, $n \in \{1, 2, 3, 4, 5\}$, $|\mu| \leq 12$;
\item $g \in \{4, 5\}$, $n \in \{1, 2, 3, 4, 5\}$, $|\mu| \leq 10$.
\end{itemize}

One can observe rich structure in the roots of the deformed monotone Hurwitz numbers and write down further conjectures. As an example, consider the following, which asserts that the roots of $\HN{g}{n}{\mu_1, \ldots, \mu_n}$ behave well for fixed $\mu_1, \ldots, \mu_n$ and increasing $g$. A great deal of data can be generated to support this conjecture, some examples of which appear in the tables of \cref{table:gincreasing1,table:gincreasing2}.

\begin{conjecture} \label{con:largeg}
Fix positive integers $\mu_1, \ldots, \mu_n$ whose sum is $d$. As $g \to \infty$, the roots of the polynomial $\HN{g}{n}{\mu_1, \ldots, \mu_n}$, in order from smallest to largest, respectively approach the values
\[
- \frac{d-2}{1}, ~~ - \frac{d-3}{2}, ~~ - \frac{d-4}{3}, ~~ \ldots, ~~ - \frac{2}{d-3}, ~~ - \frac{1}{d-2}, ~~ 0.
\]
Moreover, the convergence to a number less than $-1$ is increasing from below, while the convergence to a non-zero number greater than $-1$ is decreasing from above.
\end{conjecture}

\begin{figure} [ht!]
\begin{tabularx}{\textwidth}{cccYccY} \toprule
$(\mu_1, \ldots, \mu_n)$ & $\alpha_1$ & $\alpha_2$ & $\alpha_3$ & $\alpha_4$ & $\alpha_5$ & $\alpha_6$ \\ \midrule
$(7)$ & -5.0012679418 & -2.0003283655 & $-1$ & -0.49991792208 & -0.19994929518 & $0$ \\
$(6, 1)$ & -5.0012679418 & -2.0003283655 & $-1$ & -0.49991792208 & -0.19994929518 & $0$ \\
$(5, 2)$ & -5.0010564352 & -2.0002736067 & $-1$ & -0.49993160767 & -0.19995775151 & $0$ \\
$(5, 1, 1)$ & -5.0012326847 & -2.0003192382 & $-1$ & -0.49992020316 & -0.19995070476 & $0$ \\
$(4, 3)$ & -5.0010564383 & -2.0002736143 & $-1$ & -0.49993160576 & -0.19995775139 & $0$ \\
$(4, 2, 1)$ & -5.0010564362 & -2.0002736092 & $-1$ & -0.49993160703 & -0.19995775147 & $0$ \\
$(4, 1, 1, 1)$ & -5.0011739286 & -2.0003040277 & $-1$ & -0.49992400462 & -0.19995305387 & $0$ \\
$(3, 3, 1)$ & -5.0010564383 & -2.0002736143 & $-1$ & -0.49993160576 & -0.19995775139 & $0$ \\
$(3, 2, 2)$ & -5.0008802353 & -2.0002279852 & $-1$ & -0.49994301017 & -0.19996479678 & $0$ \\
$(3, 2, 1, 1)$ & -5.0010270659 & -2.0002660064 & $-1$ & -0.49993350723 & -0.19995892579 & $0$ \\
$(3, 1, 1, 1, 1)$ & -5.0011004921 & -2.0002850163 & $-1$ & -0.49992875605 & -0.19995599000 & $0$ \\
$(2, 2, 2, 1)$ & -5.0008802353 & -2.0002279852 & $-1$ & -0.49994301017 & -0.19996479678 & $0$ \\
$(2, 2, 1, 1, 1)$ & -5.0009781178 & -2.0002533320 & $-1$ & -0.49993667500 & -0.19996088293 & $0$ \\
$(2, 1, 1, 1, 1, 1)$ & -5.0010189060 & -2.0002638930 & $-1$ & -0.49993403543 & -0.19995925206 & $0$ \\
$(1, 1, 1, 1, 1, 1, 1)$ & -5.0010189060 & -2.000263893081 & $-1$ & -0.49993403543 & -0.19995925206 & $0$ \\ \bottomrule
\end{tabularx}
\caption{The roots $\alpha_1 \leq \alpha_2 \leq \cdots \leq \alpha_6$ of $\HN{20}{n}{\mu_1, \ldots, \mu_n}$ for all $\mu_1, \ldots, \mu_n$ satisfying $|\mu| = 7$, to 10 decimal places. Observe the proximity of the numbers in each column to $-\frac{5}{1}, -\frac{4}{2}, -\frac{3}{3}, -\frac{2}{4}, -\frac{1}{5}, 0$, as predicted by \cref{con:largeg}.}
\label{table:gincreasing1}
\end{figure}

\begin{figure} [ht!]
\begin{tabularx}{\textwidth}{cccYccY} \toprule
$g$ & $\alpha_1$ & $\alpha_2$ & $\alpha_3$ & $\alpha_4$ & $\alpha_5$ & $\alpha_6$ \\ \midrule
$10$ & $-5.041604716958$ & $-2.010612015758$ & $-1$ & $-0.4973609986225$ & $-0.1983495446670$ & $0$ \\
$11$ & $-5.028663679645$ & $-2.007345500817$ & $-1$ & $-0.4981703446630$ & $-0.1988599882007$ & $0$ \\
$12$ & $-5.019792253540$ & $-2.005088769701$ & $-1$ & $-0.4987310363066$ & $-0.1992114313684$ & $0$ \\
$13$ & $-5.013688662391$ & $-2.003527619463$ & $-1$ & $-0.4991196479078$ & $-0.1994539484474$ & $0$ \\
$14$ & $-5.009478345470$ & $-2.002446572250$ & $-1$ & $-0.4993891042376$ & $-0.1996215835335$ & $0$ \\
$15$ & $-5.006568518035$ & $-2.001697414872$ & $-1$ & $-0.4995760061286$ & $-0.1997376039891$ & $0$ \\
$16$ & $-5.004554730409$ & $-2.001177961094$ & $-1$ & $-0.4997056830732$ & $-0.1998179765971$ & $0$ \\
$17$ & $-5.003159687696$ & $-2.000817629297$ & $-1$ & $-0.4997956762061$ & $-0.1998736923107$ & $0$ \\
$18$ & $-5.002192595245$ & $-2.000567599393$ & $-1$ & $-0.4998581404112$ & $-0.19991233463311$ & $0$ \\
$19$ & $-5.001521834116$ & $-2.000394067634$ & $-1$ & $-0.4999015024988$ & $-0.1999391451575$ & $0$ \\
$20$ & $-5.001056436287$ & $-2.000273609283$ & $-1$ & $-0.4999316070356$ & $-0.1999577514750$ & $0$ \\ \bottomrule
\end{tabularx}
\caption{The roots $\alpha_1 \leq \alpha_2 \leq \cdots \leq \alpha_6$ of $\HN{g}{3}{4, 2, 1}$ for $g = 10, 11, 12, \ldots, 20$, to 12 decimal places. Observe the monotonic convergence of the numbers in each column to $-\frac{5}{1}, -\frac{4}{2}, -\frac{3}{3}, -\frac{2}{4}, -\frac{1}{5}, 0$, as predicted by \cref{con:largeg}.}
\label{table:gincreasing2}
\end{figure}

\subsubsection*{An interlacing result}

There is a rich theory concerning interlacing polynomials, as evidenced by the tome of Fisk~\cite{fis06}. The following lemma is a slight adaptation of~\cite[Lemma~1.82]{fis06}, and will be used to prove the $(g,n) = (0,1)$ and $(g,n) = (1,1)$ cases of \cref{con:realrooted,con:interlacing}.

\begin{lemma} \label{lem:interlacing}
Let $f_1(t), f_2(t), \ldots$ be a sequence of polynomials with non-negative coefficients, such that $\deg f_i = i$. Suppose that $f_1(t)$ strictly interlaces $f_2(t)$, and that the sequence of polynomials satisfies the relation
\begin{equation} \label{eq:interlacingrecursion}
f_{d+1}(t) = \ell_d(t) \, f_{d}(t) - q_d(t) \, f_{d-1}(t),
\end{equation}
where $\ell_d(t)$ is an affine function and $q_d(t)$ is a quadratic function that is positive for $t \leq 0$. Then all polynomials in the sequence are real-rooted and $f_i(t)$ strictly interlaces $f_{i+1}(t)$ for every positive integer $i$.
\end{lemma}

\begin{proof}
Suppose that $f_{d-1}(t)$ strictly interlaces $f_d(t)$. Since $f_d(t)$ has non-negative integer coefficients, its roots can be written as $0 \geq r_1 > r_2 > \cdots > r_d$. Then \cref{eq:interlacingrecursion} evaluated at one of these roots yields
\[
f_{d+1}(r_i) = - q_d(r_i) \, f_{d-1}(r_i),
\]
which implies that $f_{d+1}(r_i)$ and $f_{d-1}(r_i)$ have different signs. Since $f_{d-1}(t)$ strictly interlaces $f_d(t)$, its values at $r_1, r_2, \ldots, r_d$ must alternate in sign, leading to the following table. The entries indicate the signs of $f_{d-1}(t)$ and $f_{d+1}(t)$, evaluated at $r_1, r_2, \ldots, r_d$ and in the limit $t \to \pm\infty$.

\begin{center}
\begin{tabular}{cccccccc} \toprule
 & $\infty$ & $r_1$ & $r_2$ & $r_3$ & $\cdots$ & $r_d$ & $-\infty$ \\ \midrule
$f_{d-1}(t)$ & $+$ & $+$ & $-$ & $+$& $\cdots$ & $(-)^{d-1}$ & $(-)^{d-1}$ \\
$f_{d+1}(t)$ & $+$ & $-$ & $+$ & $-$ & $\cdots$ & $(-)^d$ & $(-)^{d+1}$ \\ \bottomrule
\end{tabular}
\end{center}

The intermediate value theorem implies that $f_{d+1}(t)$ has $d+1$ real roots and that $f_d(t)$ strictly interlaces $f_{d+1}(t)$. The result then follows by induction on $d$.
\end{proof}

\begin{proposition} \label{prop:genus01}
\cref{con:realrooted,con:interlacing} hold for $(g,n) = (0,1)$ and $(g,n) = (1,1)$.
\end{proposition}

\begin{proof}
Setting $g = 0$ in the one-point recursion of \cref{thm:onepoint} yields
\begin{equation} \label{eq:genus0point}
d^2 \HN{0}{1}{d} = (d-1) (2d-3) (t+1) \, \HN{0}{1}{d-1} - (d-2) (d-3) (t-1)^2 \, \HN{0}{1}{d-2}.
\end{equation}
Divide both sides by $d^2t$ to obtain a recursion governing the polynomials $\frac{\HN{0}{1}{1}}{t}, \frac{\HN{0}{1}{2}}{t}, \frac{\HN{0}{1}{3}}{t}, \ldots$ of the form of \cref{eq:interlacingrecursion}. It follows from \cref{lem:interlacing} that this sequence is real-rooted and strictly interlacing. Therefore, the sequence $\HN{0}{1}{1}, \HN{0}{1}{2}, \HN{0}{1}{3}, \ldots$ is a sequence of real-rooted polynomials in which each polynomial interlaces the next.

For the genus 1 case, we argue in exactly the same way, using the as yet unproven claim that
\begin{equation} \label{eq:genus1point}
d(d-2) \, \HN{1}{1}{d} = (d-1) (2d-1) (t+1) \, \HN{1}{1}{d-1} - (d-2) (d+1) (t-1)^2 \, \HN{1}{1}{d-2}.
\end{equation}

It remains to prove this relation, which we do using a generating function approach. Multiply both sides by $x^{d-1}$ and sum over all positive integers $d$. Setting $w_{1,1}(x) = \sum_{d=1}^\infty d \, \HN{1}{1}{d} \, x^{d-1}$, the claim is equivalent to
\begin{equation} \label{eq:diffeqw11}
-x \left[ (t-1)^2 x^2 - 2(t+1) x + 1 \right] \frac{\partial}{\partial x} w_{1,1}(x) = \left[ 4(t-1)^2 x^2 - 3(t+1) x - 1 \right] w_{1,1}(x).
\end{equation}

We borrow \cref{thm:TR} from the next section, which allows us to explicitly compute via $w_{1,1}(x) = \frac{\omega_{1,1}(z)}{\dd x}$ that
\begin{equation} \label{eq:w11}
w_{1,1}(x) = -\frac{t z (z-1) (1 - z + tz)^4}{(tz^2 - z^2 + 2z - 1)^5},
\end{equation}
where the coordinates $x$ and $z$ are related by $x = \frac{z(1 - z)}{1-z+tz}$. One can now check that the expression for $w_{1,1}(x)$ of \cref{eq:w11} satisfies the differential equation of \cref{eq:diffeqw11} using a computer algebra system, thereby proving \cref{eq:genus1point}.
\end{proof}

\begin{remark}
It is natural to seek higher genus analogues of \cref{eq:genus0point,eq:genus1point}, in order to extend the result and proof of \cref{prop:genus01}. However, one can prove that there exists no recursion of the form
\[
d A(d) \, \HN{2}{1}{d} = (d-1) B(d) f(t) \, \HN{2}{1}{d-1} - (d-2) C(d) g(t) \, \HN{2}{1}{d-2},
\]
where $A(d), B(d), C(d)$ are affine functions of $d$, $f(t)$ is an affine function of $t$, and $g(t)$ is a quadratic function of $t$. So the exact argument used in the proof of \cref{prop:genus01} cannot extend to genus 2 and higher without some alteration.
\end{remark}

\section{Topological recursion} \label{sec:TR}

\subsection{Topological recursion for deformed monotone Hurwitz numbers}

The topological recursion arose in the mathematical physics literature as a formalism to capture loop equations in matrix models~\cite{che-eyn06, eyn-ora07}. There is now an extensive theory around topological recursion and it is known to govern a wide range of enumerative-geometric problems beyond the original matrix model context. These include: Hurwitz numbers and various generalisations~\cite{bou-mar08, BHLM14, do-lei-nor16, eyn-mul-saf11, BDKLM22, DKPS19}, monotone Hurwitz numbers~\cite{do-dye-mat17}, lattice points in moduli spaces of curves~\cite{nor13, cha-do-mos19}, intersection theory on moduli spaces of curves~\cite{eyn-ora07, eyn14}, Gromov--Witten theory of $\mathbb{CP}^1$~\cite{nor-sco14, DOSS14}, Gromov--Witten theory of toric Calabi--Yau threefolds~\cite{BKMP09, eyn-ora15, fan-liu-zon20}, and cohomological field theories~\cite{DOSS14}. The topological recursion is also conjectured to govern certain quantum knot invariants, such as the large $N$ asymptotics of coloured Jones polynomials~\cite{dij-fuj-man11, bor-eyn15} and coloured HOMFLY-PT polynomials~\cite{GJKS15}.

In its original formulation due to Chekhov, Eynard and Orantin~\cite{che-eyn06, eyn-ora07}, the topological recursion takes as input a spectral curve $(\mathcal{C}, x, y, \omega_{0,2})$ comprising a compact Riemann surface $\mathcal{C}$ equipped with two meromorphic functions $x, y: \mathcal{C} \to \mathbb{CP}^1$ and a bidifferential $\omega_{0,2}$ with a double pole along the diagonal. These are required to satisfy mild technical assumptions, which we do not mention here. Indeed, rather than fully define the topological recursion, we refer the reader to the many existing expositions in the literature~\cite{do-dye-mat17, eyn-ora09}. For the present discussion, it suffices to note that the topological recursion uses the initial data to define the base cases $\omega_{0,1}(z_1) = y(z_1) \, \dd x(z_1)$ and $\omega_{0,2}(z_1, z_2)$, and then produces so-called correlation differentials via the following recursive formula.\footnote{The reader should note that the various appearances of the topological recursion formula in the literature differ subtly, particularly in terms of the expression for $\omega_{0,1}$, the definition of the recursion kernel $K(z_1, z)$, and the choice of sign for $y$. Ultimately, the first two differences are inconsequential and a change of sign for $y$ simply sends each correlation differential $\omega_{g,n}$ to $(-1)^n \omega_{g,n}$.}
\begin{align*}
\omega_{g,n}(z_1, \ldots, z_n) = \sum_{\alpha} \mathop{\mathrm{Res}}_{z = \alpha} K(z_1, z) \Bigg[ &\omega_{g-1,n+1}(z, \sigma_\alpha(z), z_2, \ldots, z_n) \\
&+ \sum_{\substack{g_1+g_2=g \\ I_1 \sqcup I_2 = \{2, \ldots, n\}}}^\circ \omega_{g_1,|I_1|+1}(z, z_{I_1}) \, \omega_{g_2,|I_2|+1}(\sigma_\alpha(z), z_{I_2}) \Bigg]
\end{align*}
In various settings, the correlation differential $\omega_{g,n}$ acts as a generating function for enumerative-geometric quantities, where $g$ represents the genus of a surface and $n$ its number of punctures, or boundaries.

It was previously shown that the monotone Hurwitz numbers are governed by the topological recursion in the following sense~\cite{do-dye-mat17}. Topological recursion on the spectral curve\footnote{The topological recursion is not sensitive to reparametrisation of the underlying plane curve, whose global description in this case is $xy^2 - y + 1 = 0$. We have expressed the spectral curve here using a different rational parametrisation to that appearing in~\cite{do-dye-mat17}, in order to align more closely with the statement and proof of \cref{thm:TR}.}
\[
\mathcal{C} = \mathbb{CP}^1, \quad x(z) = z(1 - z), \quad y(z) = \frac{1}{1 - z}, \quad \omega_{0,2}(z_1, z_2) = \frac{\dd z_1 \, \dd z_2}{(z_1 - z_2)^2}
\]
produces correlation differentials satisfying
\[
\omega_{g,n}(z_1, \ldots, z_n) = \dd_1 \cdots \dd_n \sum_{\mu_1, \ldots, \mu_n=1}^{\infty} \vec{H}_{g,n}(\mu_1, \ldots, \mu_n) \, x(z_1)^{\mu_1} \cdots x(z_n)^{\mu_n} + \delta_{g,0} \delta_{n,2} \, \frac{\dd x(z_1) \, \dd x(z_2)}{\left( x(z_1) - x(z_2) \right)^2}.
\]
Here, $\dd_i$ represents the exterior derivative in the $i$th coordinate and $\vec{H}_{g,n}(\mu_1, \ldots, \mu_n) = \big[ \HN{g}{n}{\mu_1, \ldots, \mu_n} \big]_{t=1}$ denotes a monotone Hurwitz number.

It is thus natural to consider whether the deformed monotone Hurwitz numbers satisfy the topological recursion on a deformation of the spectral curve above. Perhaps unsurprisingly, this is indeed the case and follows from the representation-theoretic interpretation for deformed monotone Hurwitz numbers of \cref{cor:characterformula} and a powerful result of Bychkov, Dunin-Barkowski, Kazarian and Shadrin~\cite{BDKS21}, which builds on previous work of Alexandrov, Chapuy, Eynard and Harnad~\cite{ACEH20}.

\begin{theorem} \label{thm:TR}
Topological recursion on the spectral curve
\[
\mathcal{C} = \mathbb{CP}^1, \quad x(z) = \frac{z(1 - z)}{1-z+tz}, \quad y(z) = \frac{1-z+tz}{1 - z}, \quad \omega_{0,2}(z_1, z_2) = \frac{\dd z_1 \, \dd z_2}{(z_1 - z_2)^2}
\]
produces correlation differentials satisfying
\[
\omega_{g,n}(z_1, \ldots, z_n) = \dd_1 \cdots \dd_n \sum_{\mu_1, \ldots, \mu_n=1}^{\infty} \HN{g}{n}{\mu_1, \ldots, \mu_n} \, x(z_1)^{\mu_1} \cdots x(z_n)^{\mu_n} + \delta_{g,0} \delta_{n,2} \, \frac{\dd x(z_1) \, \dd x(z_2)}{\left( x(z_1) - x(z_2) \right)^2}.
\]
\end{theorem}

\begin{proof}
The main theorem of Bychkov, Dunin-Barkowski, Kazarian and Shadrin in~\cite{BDKS21} states that the topological recursion governs the coefficients of certain KP tau functions of hypergeometric type. We present here a simplified version of their result that is fit for purpose.

Let $\tilde{y}(z) = \sum_{i=1}^\infty y_i \, z^i$ be the series expansion of a rational function satisfying $\tilde{y}(0) = 0$ and let $f(z)$ be a rational function satisfying $f(0) = 1$. From this data, define the generating function
\[
Z(p_1, p_2, \ldots; \h) = \sum_{\lambda \in \mathcal{P}} s_\lambda(p_1, p_2, \ldots) \, s_\lambda \Big( \frac{y_1}{\h}, \frac{y_2}{\h}, \ldots \Big) \prod_{\square \in \lambda} f(\h \, c(\square)),
\]
where $\mathcal{P}$ is the set of partitions, $s_\lambda$ represents the Schur function expressed in terms of power sum symmetric functions, and $c(\square)$ denotes the content of the box $\square$ in the Young diagram of a partition. This is a KP tau function of hypergeometric type and there exist ``weighted Hurwitz numbers'' $X_{g,n}(\mu_1, \ldots, \mu_n)$ for $g \geq 0$, $n \geq 1$ and $\mu_1, \ldots, \mu_n \geq 1$ such that
\begin{equation} \label{eq:partitionexpansion}
Z(p_1, p_2, \ldots; \h) = \exp \bigg( \sum_{g=0}^\infty \sum_{n=1}^\infty \frac{\h^{2g-2+n}}{n!} \sum_{\mu_1, \ldots, \mu_n=1}^\infty X_{g,n}(\mu_1, \ldots, \mu_n) \, p_{\mu_1} \cdots p_{\mu_n} \bigg).
\end{equation}

Bychkov, Dunin-Barkowski, Kazarian and Shadrin~\cite{BDKS21} prove that topological recursion on the spectral curve
\[
\mathcal{C} = \mathbb{CP}^1, \quad x(z) = \frac{z}{f(\tilde{y}(z))}, \quad y(z) = \frac{\tilde{y}(z)}{x(z)}, \quad \omega_{0,2}(z_1, z_2) = \frac{\dd z_1 \, \dd z_2}{(z_1 - z_2)^2}
\]
produces correlation differentials satisfying
\[
\omega_{g,n}(z_1, \ldots, z_n) = \dd_1 \cdots \dd_n \sum_{\mu_1, \ldots, \mu_n=1}^{\infty} X_{g,n}(\mu_1, \ldots, \mu_n) \, x(z_1)^{\mu_1} \cdots x(z_n)^{\mu_n} + \delta_{g,0} \delta_{n,2} \, \frac{\dd x(z_1) \, \dd x(z_2)}{\left( x(z_1) - x(z_2) \right)^2}.
\]

We simply specialise this result to $\tilde{y}(z) = z$ and $f(z) = \frac{1-z+tz}{1-z}$. It remains to check that the monotone deformed Hurwitz numbers agree with the weighted Hurwitz numbers with this particular choice of data.

The exponential of \cref{eq:partitionexpansion} transforms the connected generating function to the disconnected generating function, so it suffices to show that the disconnected deformed monotone Hurwitz numbers satisfy
\[
\dHN{g}{n}{\mu_1, \ldots, \mu_n} = |\mathrm{Aut}(\mu)| \, \big[ \h^{2g-2+n} p_{\mu_1} \cdots p_{\mu_n} \big] Z(p_1, p_2, \ldots; \h).
\]
Here, $\mathrm{Aut}(\mu)$ denotes the subgroup of permutations in $S_n$ that fix $(\mu_1, \ldots, \mu_n)$. This automorphism factor arises because the summation in \cref{eq:partitionexpansion} is over all tuples of positive integers, rather than partitions.

Now use $s_\lambda(p_1, p_2, \ldots) = \displaystyle\sum_{\mu \vdash |\lambda|} \frac{\chi^\lambda_\mu}{|\mathrm{Aut}(\mu)|} \prod \frac{p_{\mu_i}}{\mu_i}$ in the expression
\[
Z(p_1, p_2, \ldots; \h) = \sum_{\lambda \in \mathcal{P}} s_\lambda(p_1, p_2, \ldots) \, s_\lambda ( \tfrac{1}{\h}, 0, 0, \ldots ) \prod_{\square \in \lambda} \frac{1 - \h \, c(\square) + t \h \, c(\square)}{1 - \h \, c(\square)}.
\]
to obtain
\begin{multline*}
|\mathrm{Aut}(\mu)| \, \big[ \h^{2g-2+n} p_{\mu_1} \cdots p_{\mu_n} \big] Z(p_1, p_2, \ldots; \h) \\
= \frac{1}{\mu_1 \cdots \mu_n} \, \big[ \h^{|\mu|+2g-2+n} \big] \sum_{\lambda \vdash |\mu|} \frac{\chi^\lambda_{11 \cdots 1} \, \chi^\lambda_\mu}{|\mu|!} \prod_{\square \in \lambda} \frac{1 - \h \, c(\square) + t \h \, c(\square)}{1 - \h \, c(\square)}.
\end{multline*}
That this is equal to $\dHN{g}{n}{\mu_1, \ldots, \mu_n}$ is precisely the content of \cref{prop:reptheory}, and this concludes the proof.
\end{proof}

\begin{remark}
The meromorphic functions $x, y: \mathbb{CP}^1 \to \mathbb{CP}^1$ of \cref{thm:TR} satisfy
\[
xy^2 + (t-1)xy - y + 1 = 0,
\]
which is the global form of the spectral curve. The general theory of topological recursion asserts that, for $(g,n) \neq (0,1)$ or $(0,2)$, the correlation differential $\omega_{g,n}$ is a rational multidifferential on this curve with poles only at the zeroes of $\dd x$ and further conditions on the pole structure~\cite{eyn-ora07}. This in turn leads to a polynomiality structure theorem for the deformed monotone Hurwitz numbers. Furthermore, the relation between topological recursion and cohomological field theories should lead to an interpretation of deformed monotone Hurwitz numbers as intersection numbers on moduli spaces of curves~\cite{eyn14, DOSS14}. However, we do not pursue this investigation further in the present work.
\end{remark}

\subsection{Topologising sequences of polynomials}

\subsubsection*{From Catalan curves to Narayana curves}

As previously mentioned in \cref{eq:hurwitznarayana}, the $(g,n) = (0,1)$ case of the deformed monotone Hurwitz numbers recovers the Narayana polynomials in the following way.
\[
(\mu+1) \, \HN{0}{1}{\mu+1} = \mathrm{Nar}_{\mu}(t) := \sum_{i=1}^{\mu} \frac{1}{\mu} \binom{\mu}{i} \binom{\mu}{i-1} \, t^i
\]
Hence, we consider the deformed monotone Hurwitz numbers to be a ``topological generalisation'' of the Narayana polynomials. This viewpoint appears throughout the literature, such as in Dumitrescu and Mulase's exposition on generalised Catalan numbers~\cite{dum-mul18}.

We propose that the topological recursion can be used as a mechanism to ``topologise'' sequences of polynomials. One would apply topological recursion to a spectral curve with a deformation parameter~$t$ to obtain a family of polynomials $X_{g,n}^t(\mu_1, \ldots, \mu_n)$ for $g \geq 0$, $n \geq 1$ and $\mu_1, \ldots, \mu_n \geq 1$ such that $X^t_{0,1}(1), X^t_{0,1}(2), \ldots$ recovers the original sequence of polynomials in $t$. Furthermore, we claim that interesting properties of the original sequence of polynomials can persist into the topological generalisation. Examples of such properties are the symmetry and unimodality of coefficients as well as the interlacing of roots. These both hold for Narayana polynomials and are respectively proven (\cref{prop:symmetry}) and conjectured (\cref{con:interlacing}) to hold for deformed monotone Hurwitz numbers more generally.

There is more than one way to topologise a given sequence of polynomials. In order to obtain a topological generalisation of the Narayana numbers, it is natural to consider spectral curves that store Catalan numbers in the correlation differential $\omega_{0,1}$ and to consider a suitable $t$-deformation. The literature suggests three natural candidates for such ``Catalan spectral curves'' and these curves govern monotone Hurwitz numbers, map enumeration, and dessin d'enfant enumeration. This leads to the following three ``Narayana spectral curves'' with $\mathcal{C} = \mathbb{CP}^1$, $\omega_{0,2} = \frac{\dd z_1 \, \dd z_2}{(z_1 - z_2)^2}$, and the meromorphic functions $x, y: \mathcal{C} \to \mathbb{CP}^1$ satisfying the following.
\begin{itemize}
\item {\em Monotone Hurwitz numbers.} We have
\[
y(x) = \sum_{\mu=0}^\infty \mathrm{Cat}_\mu \, x^\mu \quad \Rightarrow \quad x y^2 - y + 1 = 0,
\]
which is the global form of the spectral curve for monotone Hurwitz numbers~\cite{do-dye-mat17}. Promoting the Catalan numbers to the Narayana polynomials, one obtains
\[
y(x) = \sum_{\mu=0}^\infty \mathrm{Nar}_\mu(t) \, x^\mu \quad \Rightarrow \quad xy^2 + (t-1)xy - y + 1 = 0.
\]
This is the global form of the spectral curve for deformed monotone Hurwitz numbers, which appeared above in \cref{thm:TR}.

\item {\em Map enumeration.} We have
\[
y(x) = \sum_{\mu=0}^\infty \mathrm{Cat}_\mu \, x^{-2\mu-1} \quad \Rightarrow \quad y^2 - xy + 1 = 0,
\]
which is the global form of the spectral curve for the enumeration of maps~\cite{che-eyn06, eyn-ora07}. Promoting the Catalan numbers to the Narayana polynomials, one obtains
\[
y(x) = \sum_{\mu=0}^\infty \mathrm{Nar}_\mu(t) \, x^{-2\mu-1} \quad \Rightarrow \quad xy^2 - x^2y + (t-1) y + x = 0.
\]
To the best of our knowledge, this is the global form of a spectral curve that has not yet been studied in the literature. We do not investigate it further in the present work, largely due to the fact that it has genus 1 for generic values of $t$. The topological recursion on spectral curves of positive genus is much less straightforward than on those of genus 0. Nevertheless, it would be interesting to understand this spectral curve and whether it governs a natural enumerative problem.

\item {\em Dessin d'enfant enumeration.} We have
\[
y(x) = \sum_{\mu=0}^\infty \mathrm{Cat}_\mu \, x^{-\mu-1} \quad \Rightarrow \quad x y^2 - x y + 1 = 0,
\]
which is the global form of the spectral curve for the enumeration of dessins d'enfant~\cite{do-nor18, kaz-zog15}. Promoting the Catalan numbers to the Narayana polynomials, one obtains
\[
y(x) = \sum_{\mu=0}^\infty \mathrm{Nar}_\mu(t) \, x^{-\mu-1} \quad \Rightarrow \quad x y^2 - x y + (t-1)y + 1 = 0.
\]
This is the global form of a spectral curve with genus 0. It leads to a topological generalisation of the Narayana polynomials that differs from the deformed monotone Hurwitz numbers. We will consider its properties in the next section.
\end{itemize}

\begin{remark}
We certainly do not claim that the spectral curves above provide the only natural topological generalisations of the Catalan numbers or the Narayana polynomials. Indeed, one could study spectral curves arising from $y(x) = \sum_{\mu=0}^\infty \mathrm{Cat}_\mu \, x^{a\mu+b}$ or $y(x) = \sum_{\mu=0}^\infty \mathrm{Nar}_\mu(t) \, x^{a\mu+b}$ for judicious choices of $a$ and~$b$, and it is not unlikely that these would lead to interesting enumerative problems.
\end{remark}

\subsubsection*{Deforming the dessin d'enfant enumeration}

The previous discussion motivates the study of the spectral curve whose global form is
$x y^2 - x y + (t-1)y + 1 = 0$. In parametrised form, this corresponds to the spectral curve
\[
\mathcal{C} = \mathbb{CP}^1, \quad x(z) = \frac{1-z+tz}{z(1-z)}, \quad y(z) = z, \quad \omega_{0,2}(z_1, z_2) = \frac{\dd z_1 \, \dd z_2}{(z_1 - z_2)^2}.
\]
By construction, this spectral curve should govern a $t$-deformation of the usual dessin d'enfant enumeration. In the following, we prove that this enumeration arises by attaching a multiplicative weight $t$ to each black vertex of a dessin d'enfant. Moreover, we observe that the polynomials arising from this construction empirically satisfy real-rootedness and interlacing properties analogous to those observed for deformed monotone Hurwitz numbers in \cref{con:realrooted,con:interlacing}. Although the weighted enumeration of dessins d'enfant has been studied previously~\cite{kaz-zog15}, these real-rootedness and interlacing properties have not been observed in the literature, to the best of our knowledge. Our discussion of the weighted enumeration of dessins d'enfant should be considered as a further case study promoting the viewpoint that topological recursion produces interesting topological generalisations of sequences of polynomials.

First, let us define dessins d'enfant --- in other words, bicoloured maps --- and their enumeration.

\begin{definition} A {\em map} is a finite graph --- possibly with loops or multiple edges --- embedded in a compact oriented surface such that the complement of the graph is a disjoint union of topological disks. We refer to these disks as {\em faces} and require that they are labelled $1, 2, 3, \ldots, n$, where $n$ denotes the number of faces.

A {\em dessin d'enfant} is a map whose vertices have been coloured black or white such that each edge is incident to one black vertex and one white vertex. The {\em degree} of a face is defined to be half the number of edges incident to it.

An {\em equivalence} between two maps (or dessins d'enfant) is a bijection between their respective vertices, edges and faces that is realised by an orientation-preserving homeomorphism of their underlying surfaces and preserves all adjacencies, incidences, labels (and colours). An {\em automorphism} of a map (or dessin d'enfant) is an equivalence from the map (or dessin d'enfant) to itself.
\end{definition}

\begin{definition} \label{def:dessincount}
For $g \geq 0$, $n \geq 1$ and $\mu_1, \ldots, \mu_n \geq 1$, let $D^t_{g,n}(\mu_1, \ldots, \mu_n) \in \mathbb{Q}[t]$ denote the weighted enumeration of connected genus $g$ dessins d'enfant with $n$ labelled faces of degrees $\mu_1, \ldots, \mu_n$. The weight attached to a dessin d'enfant $\Gamma$ is $\frac{t^{b(\Gamma)}}{|\mathrm{Aut}(\Gamma)|}$, where $b(\Gamma)$ denotes the number of black vertices of $\Gamma$ and $\mathrm{Aut}(\Gamma)$ denotes the automorphism group of $\Gamma$. Let $D^{t\bullet}_{g,n}(\mu_1, \ldots, \mu_n) \in \mathbb{Q}[t]$ denote the analogous count, including possibly disconnected dessins d'enfant.
\end{definition}

\begin{example}
Consider the dessin d'enfant enumeration $D^t_{0,2}(2,2)$. The only dessins d'enfant that contribute are the five pictured below. We have represented these as plane graphs, with the face labelled 1 corresponding to the bounded region and the face labelled 2 corresponding to the unbounded region.

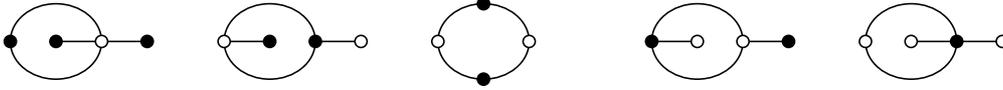
\begin{center}
\begin{tikzpicture} [line width = 0.6pt]
\def\x{0.6}
\def\y{0.5}
\def\r{0.08}
\begin{scope}
	\draw (0,0) circle [x radius = \x, y radius = \y];
	\draw (0,0) -- (2*\x,0);
	\filldraw[fill=black] (-\x,0) circle (\r);
	\filldraw[fill=black] (0,0) circle (\r);
	\filldraw[fill=white] (\x,0) circle (\r);
	\filldraw[fill=black] (2*\x,0) circle (\r);
\end{scope}
\begin{scope} [xshift = 80]
	\draw (0,0) circle [x radius = \x, y radius = \y];
	\draw (-\x,0) -- (0,0);
	\draw (\x,0) -- (2*\x,0);
	\filldraw[fill=white] (-\x,0) circle (\r);
	\filldraw[fill=black] (0,0) circle (\r);
	\filldraw[fill=black] (\x,0) circle (\r);
	\filldraw[fill=white] (2*\x,0) circle (\r);
\end{scope}
\begin{scope} [xshift = 160]
	\draw (0,0) circle [x radius = \x, y radius = \y];
	\filldraw[fill=white] (-\x,0) circle (\r);
	\filldraw[fill=white] (\x,0) circle (\r);
	\filldraw[fill=black] (0,-\y) circle (\r);
	\filldraw[fill=black] (0,\y) circle (\r);
\end{scope}
\begin{scope} [xshift = 240]
	\draw (0,0) circle [x radius = \x, y radius = \y];
	\draw (-\x,0) -- (0,0);
	\draw (\x,0) -- (2*\x,0);
	\filldraw[fill=black] (-\x,0) circle (\r);
	\filldraw[fill=white] (0,0) circle (\r);
	\filldraw[fill=white] (\x,0) circle (\r);
	\filldraw[fill=black] (2*\x,0) circle (\r);
\end{scope}
\begin{scope} [xshift = 320]
	\draw (0,0) circle [x radius = \x, y radius = \y];
	\draw (0,0) -- (2*\x,0);
	\filldraw[fill=white] (-\x,0) circle (\r);
	\filldraw[fill=white] (0,0) circle (\r);
	\filldraw[fill=black] (\x,0) circle (\r);
	\filldraw[fill=white] (2*\x,0) circle (\r);
\end{scope}
\end{tikzpicture}
\end{center}

The central dessin d'enfant has two automorphisms, while the remaining dessins d'enfant have one, so we have $D^t_{0,2}(2,2) = \big( \frac{t^3}{1} \big) + \big( \frac{t^2}{1} + \frac{t^2}{2} + \frac{t^2}{1} \big) + \big( \frac{t^1}{1} \big) = t^3 + \frac{5}{2} t^2 + t$.
\end{example}

\begin{proposition}
If $|\mu| < 2g + n$, then $D_{g,n}^t(\mu_1, \ldots, \mu_n) = 0$. If $|\mu| \geq 2g + n$, then $D_{g,n}^t(\mu_1, \ldots, \mu_n) \in \mathbb{Q}[t]$ is a polynomial of degree $|\mu| + 1 - 2g - n$ whose coefficients are symmetric.
\end{proposition}

\begin{proof}
Consider a dessin d'enfant with genus $g$ and $n$ faces with degrees $\mu_1, \ldots, \mu_n$. By an Euler characteristic calculation, the number of vertices is $|\mu| + 2 - 2g - n$. Since there must be at least one black vertex and at least one white vertex, we must have have $|\mu| \geq 2g + n$.

We will show that, under the assumption $|\mu| \geq 2g + n$, there exists a dessin d'enfant with exactly one white vertex. Take a polygon with $4g+2$ sides, whose vertices are alternately coloured black and white. By identifying opposite edges, one obtains a dessin d'enfant on a genus $g$ surface, with one black vertex, one white vertex, and one face with degree $2g+1$. By adding $n-1$ appropriate edges from the black vertex to the white vertex, one can obtain a genus $g$ dessin d'enfant with one black vertex, one white vertex, and $n$~faces with any degrees $\mu_1, \ldots, \mu_n$ satisfying $|\mu| = 2g + n$. We can relax this condition to $|\mu| \geq 2g + n$ by adding appropriate black vertices with degree 1 that are adjacent to the unique white vertex. Such a dessin d'enfant contributes to degree $|\mu| + 1 - 2g - n$ in $D_{g,n}^t(\mu_1, \ldots, \mu_n)$, and there can be no contributions in higher degree.

Finally, observe that switching the colours of the vertices of a dessin d'enfant changes the number of black vertices from $b$ to $(|\mu| + 2 - 2g - n) - b$, leading to the desired symmetry in the coefficients of $D_{g,n}^t(\mu_1, \ldots, \mu_n)$.
\end{proof}

There is an analogue of the cut-and-join recursion for deformed monotone Hurwitz numbers of \cref{thm:cutjoin} in the context of the dessin d'enfant enumeration. We omit the proof, since it can be found in the literature~\cite{kaz-zog15}.

\begin{proposition} [Cut-and-join recursion] \label{prop:dessincutjoin}
Apart from the initial condition $D^t_{0,1}(1) = 1$, the dessin d'enfant enumeration satisfies
\begin{align*}
\mu_1 D^t_{g,n}(\mu_1, \mu_S) &= \sum_{i=2}^n (\mu_1+\mu_i-1) \, D^t_{g,n-1}(\mu_1+\mu_i-1, \mu_{S \setminus \{i\}}) + (t+1) (\mu_1-1) \, D^t_{g,n}(\mu_1-1, \mu_S) \\
&+ \sum_{\alpha + \beta = \mu_1-1} \alpha \beta \Bigg[ D^t_{g-1,n+1}(\alpha, \beta, \mu_S) + \sum_{\substack{g_1+g_2=g \\ I_1 \sqcup I_2 = S}} D^t_{g_1,|I_1|+1}(\alpha, \mu_{I_1}) \, D^t_{g_2,|I_2|+1}(\beta, \mu_{I_2}) \Bigg],
\end{align*}
where we use the notation $S = \{2, 3, \ldots, n\}$ and $\mu_I = \{\mu_{i_1}, \mu_{i_2}, \ldots, \mu_{i_k}\}$ for $I = \{i_1, i_2, \ldots, i_k\}$.
\end{proposition}

It is well-known that dessins d'enfant correspond to pairs of permutations acting on the edges --- one permutation rotates edges around black vertices and the other rotates edges around white vertices. (Equivalent descriptions in the literature often use triples of permutations that compose to give the identity.) This leads to an expression for the disconnected enumeration of dessins d'enfant in terms of characters of the symmetric group, via Frobenius's formula. For details, we recommend the book of Lando and Zonkin~\cite{lan-zvo04}. To incorporate the parameter $t$, we observe that its exponent should equal the number of cycles in the permutation that rotates edges around black vertices. As a consequence of Jucys's results described in part (a) of \cref{prop:jucys}, we obtain the following result, for which we omit the proof.

\begin{proposition} \label{prop:dessinsrep}
The disconnected enumeration of dessins d'enfant is given by the formula
\[
D_{g,n}^{t\bullet}(\mu_1, \ldots, \mu_n) = \frac{1}{\mu_1 \cdots \mu_n} \, \big[ \h^{2g-2+n} \big] \sum_{\lambda \in \mathcal{P}} \chi^\lambda_\mu \, s_\lambda(\tfrac{1}{\h}, \tfrac{1}{\h}, \ldots) \prod_{\square \in \lambda} \left( t + \h \, c(\square) \right).
\]
\end{proposition}

We are now in a position to prove that topological recursion governs 	the weighted enumeration of dessins d'enfant.\footnote{Topological recursion for a more general weighted enumeration of dessins d'enfant was previously studied by Kazarian and Zograf~\cite{kaz-zog15}, although they obtain a different form for the spectral curve.}

\begin{theorem} \label{thm:dessinTR}
Topological recursion on the spectral curve
\[
\mathcal{C} = \mathbb{CP}^1, \quad x(z) = \frac{1 - z + tz}{z(1-z)}, \quad y(z) = z, \quad \omega_{0,2} = \frac{\dd z_1 \, \dd z_2}{(z_1-z_2)^2}
\]
produces correlation differentials satisfying
\[
\omega_{g,n}(z_1, \ldots, z_n) = \dd_1 \cdots \dd_n \sum_{\mu_1, \ldots, \mu_n=1}^{\infty} D^t_{g,n}(\mu_1, \ldots, \mu_n) \, x(z_1)^{-\mu_1} \cdots x(z_n)^{-\mu_n} + \delta_{g,0} \delta_{n,2} \, \frac{\dd x(z_1) \, \dd x(z_2)}{\left( x(z_1) - x(z_2) \right)^2}.
\]
\end{theorem}

\begin{proof}
We again invoke the result of Bychkov, Dunin-Barkowski, Kazarian and Shadrin~\cite{BDKS21}, noting that the more restricted result of Alexandrov, Chapuy, Eynard and Harnad~\cite{ACEH20} would suffice in this particular context. See the proof of \cref{thm:TR} above for a statement of the result.

The representation-theoretic interpretation of the dessin d'enfant enumeration given in \cref{prop:dessinsrep} implies that they are weighted Hurwitz numbers with the choice of data $\tilde{y}(z) = \sum_{i=1}^\infty z^i = \frac{z}{1-z}$ and $f(z) = t + z$. Hence, topological recursion on the spectral curve
\[
\mathcal{C} = \mathbb{CP}^1, \quad x(z) = \frac{z}{f(\tilde{y}(z))} = \frac{z(1-z)}{t-tz+z}, \quad y(z) = \frac{\tilde{y}(z)}{x(z)} = \frac{t-tz+z}{(1-z)^2}, \quad \omega_{0,2}(z_1, z_2) = \frac{\dd z_1 \, \dd z_2}{(z_1 - z_2)^2}
\]
produces correlation differentials satisfying
\[
\omega_{g,n}(z_1, \ldots, z_n) = \dd_1 \cdots \dd_n \sum_{\mu_1, \ldots, \mu_n=1}^{\infty} D^t_{g,n}(\mu_1, \ldots, \mu_n) \, x(z_1)^{\mu_1} \cdots x(z_n)^{\mu_n} + \delta_{g,0} \delta_{n,2} \, \frac{\dd x(z_1) \, \dd x(z_2)}{\left( x(z_1) - x(z_2) \right)^2}.
\]

This does not yet match our spectral curve, since the enumeration is stored as an expansion at $x(z) = 0$ rather than at $x(z) = \infty$. We obtain the desired spectral curve by performing the following three transformations in order:
\[
(x, y, t) \mapsto (x^{-1}, x^2y, t), \qquad (x, y, t) \mapsto (x, y+tx, t), \qquad (x, y, t) (tx, y, t^{-1}).
\]
The first transformation changes the expansion at $x(z) = 0$ to an expansion at $x(z) = \infty$; the second transformation preserves the correlation differentials~\cite[Section~4.2]{eyn-ora09}; and the third transformation preserves the coefficients of the expansion, due to the symmetry of the coefficients of $D^t_{g,n}(\mu_1, \ldots, \mu_n)$ and the homogeneity property of topological recursion~\cite[Section~4.1]{eyn-ora09}. Thus, we arrive at the desired result.
\end{proof}

Although the enumeration of dessins d'enfant has been studied previously~\cite{kaz-zog15}, the following conjectural properties have not yet been observed in the literature, to the best of our knowlege.

\begin{conjecture} [Real-rootedness and interlacing] \label{con:dessins}
For all $g \geq 0$, $n \geq 1$ and $\mu_1, \ldots, \mu_n \geq 1$, the dessin d'enfant enumeration $D_{g,n}^t(\mu_1, \mu_2, \ldots, \mu_n)$ is a real-rooted polynomial in $t$. Furthermore, the polynomial $D_{g,n}^t(\mu_1, \mu_2, \ldots, \mu_n)$ interlaces each of the $n$ polynomials
\[
D_{g,n}^t(\mu_1+1, \mu_2, \ldots, \mu_n), \quad D_{g,n}^t(\mu_1, \mu_2+1, \ldots, \mu_n), \quad \ldots, \quad D_{g,n}^t(\mu_1, \mu_2, \ldots, \mu_n+1).
\]
\end{conjecture}

As with the deformed monotone Hurwitz numbers, we can effectively calculate $D_{g,n}^t(\mu_1, \ldots, \mu_n)$ using the cut-and-join recursion of \cref{prop:dessincutjoin} and explicitly check for real-rootedness and interlacing. Again, one finds a wealth of data to support \cref{con:dessins}. We consider this as evidence that the real-rootedness and interlacing observed for deformed monotone Hurwitz numbers is not simply a sporadic artefact, but may be a more widespread phenomenon with a deep reason underlying it.

\begin{remark}
As previously mentioned, dessins d'enfant correspond to pairs of permutations acting on the edges, or equivalently, triples of permutations that compose to give the identity. It follows from \cref{prop:jucys} that each permutation has a unique {\em strictly monotone factorisation}, defined analogously to a monotone factorisation in \cref{def:monotone}, but with the stronger monotonicity requirement that $b_1 < b_2 < \cdots < b_r$. Using this result on the permutation that rotates edges around black vertices leads to the fact that the enumeration of dessins d'enfant can be interpreted as strictly monotone Hurwitz numbers.\end{remark}

\appendix

\section{Data} \label{app:data}

\subsubsection*{Deformed monotone Hurwitz numbers}

The deformed monotone Hurwitz numbers can be computed using the cut-and-join recursion or the topological recursion. These were both implemented in SageMath to produce the following table~\cite{sagemath}.

\begin{center}
\begin{tabularx}{\textwidth}{cXX} \toprule
$(\mu_1, \ldots, \mu_n)$ & $\mu_1 \cdots \mu_n \, \HN{0}{n}{\mu_1, \ldots, \mu_n}$ & $\mu_1 \cdots \mu_n \, \HN{1}{n}{\mu_1, \ldots, \mu_n}$ \\ \midrule
$(1)$ & $1$ & $0$ \\
$(2)$ & $t$ & $t$ \\
$(3)$ & $t^2 + t$ & $5t^2 + 5t$ \\
$(4)$ & $t^3 + 3t^2 + t$ & $15t^3 + 40t^2 + 15t$ \\
$(5)$ & $t^4 + 6t^3 + 6t^2 + t$ & $35t^4 + 175t^3 + 175t^2 + 35t$ \\
$(6)$ & $t^5 + 10t^4 + 20t^3 + 10t^2 + t$ & $70t^5 + 560t^4 + 1050t^3 + 560t^2 + 70t$ \\ \midrule
$(1,1)$ & $t$ & $t$ \\
$(2,1)$ & $2t^2 + 2t$ & $10t^2 + 10t$ \\
$(3,1)$ & $3t^3 + 9t^2 + 3t$ & $45t^3 + 120t^2 + 45t$ \\
$(2,2)$ & $4t^3 + 10t^2 + 4t$ & $50t^3 + 128t^2 + 50t$ \\
$(4,1)$ & $4t^4 + 24t^3 + 24t^2 + 4t$ & $140t^4 + 700t^3 + 700t^2 + 140t$ \\
$(3,2)$ & $6t^4 + 30t^3 + 30t^2 + 6t$ & $168t^4 + 792t^3 + 792t^2 + 168t$ \\
$(5,1)$ & $5t^5 + 50t^4 + 100t^3 + 50t^2 + 5t$ & $350t^5 + 2800t^4 + 5250t^3 + 2800t^2 + 350t$ \\
$(4,2)$ & $8t^5 + 68t^4 + 128t^3 + 68t^2 + 8t$ & $448t^5 + 3348t^4 + 6128t^3 + 3348t^2 + 448t$ \\
$(3,3)$ & $9t^5 + 72t^4 + 138t^3 + 72t^2 + 9t$ & $462t^5 + 3432t^4 + 6312t^3 + 3432t^2 + 462t$ \\ \midrule
$(1,1,1)$ & $4t^2 + 4t$ & $20t^2 + 20t$ \\
$(2,1,1)$ & $10t^3 + 28t^2 + 10t$ & $140t^3 + 368t^2 + 140t$ \\
$(3,1,1)$ & $18t^4 + 102t^3 + 102t^2 + 18t$ & $588t^4 + 2892t^3 + 2892t^2 + 588t$ \\
$(2,2,1)$ & $24t^4 + 120t^3 + 120t^2 + 24t$ & $672t^4 + 3168t^3 + 3168t^2 + 672t$ \\
$(4,1,1)$ & $28t^5 + 268t^4 + 528t^3 + 268t^2 + 28t$ & $1848t^5 + 14548t^4 + 27128t^3 + 14548t^2 + 1848t$ \\
$(3,2,1)$ & $42t^5 + 348t^4 + 660t^3 + 348t^2 + 42t$ & $2268t^5 + 16908t^4 + 31008t^3 + 16908t^2 + 2268t$ \\
$(2,2,2)$ & $56t^5 + 424t^4 + 768t^3 + 424t^2 + 56t$ & $2688t^5 + 19128t^4 + 34416t^3 + 19128t^2 + 2688t$ \\ \bottomrule
\end{tabularx}
\end{center}

~

\begin{center}
\begin{tabularx}{\textwidth}{cXX} \toprule
$(\mu_1, \ldots, \mu_n)$ & $\mu_1 \cdots \mu_n \, \HN{2}{n}{\mu_1, \ldots, \mu_n}$ & $\mu_1 \cdots \mu_n \, \HN{3}{n}{\mu_1, \ldots, \mu_n}$ \\ \midrule
(1) & $0$ & $0$ \\
(2) & $t$ & $t$ \\
(3) & $21t^2 + 21t$ & $85t^2 + 85t$ \\
(4) & $161t^3 + 413t^2 + 161t$ & $1555t^3 + 3930t^2 + 1555t$ \\
(5) & $777t^4 + 3612t^3 + 3612t^2 + 777t$ & $14575t^4 + 65505t^3 + 65505t^2 + 14575t$ \\ \midrule
(1, 1) & $t$ & $t$ \\
(2, 1) & $42t^2 + 42t$ & $170t^2 + 170t$ \\
(3, 1) & $483t^3 + 1239t^2 + 483t$ & $4665t^3 + 11790t^2 + 4665t$ \\
(2, 2) & $504t^3 + 1278t^2 + 504t$ & $4750t^3 + 11956t^2 + 4750t$ \\
(4, 1) & $3108t^4 + 14448t^3 + 14448t^2 + 3108t$ & $58300t^4 + 262020t^3 + 262020t^2 + 58300t$ \\
(3, 2) & $3402t^4 + 15450t^3 + 15450t^2 + 3402t$ & $61116t^4 + 271764t^3 + 271764t^2 + 61116t$ \\ \midrule
(1, 1, 1) & $84t^2 + 84t$ & $340t^2 + 340t$ \\
(2, 1, 1) & $1470t^3 + 3756t^2 + 1470t$ & $14080t^3 + 35536t^2 + 14080t$ \\
(3, 1, 1) & $12726t^4 + 58794t^3 + 58794t^2 + 12726t$ & $236016t^4 + 1057824t^3 + 1057824t^2 + 236016t$ \\
(2, 2, 1) & $13608t^4 + 61800t^3 + 61800t^2 + 13608t$ & $244464t^4 + 1087056t^3 + 1087056t^2 + 244464t$ \\ \bottomrule
\end{tabularx}
\end{center}

\subsubsection*{Dessin d'enfant enumeration}

The weighted dessin d'enfant enumeration can be computed using the cut-and-join recursion or the topological recursion. These were both implemented in SageMath to produce the following table~\cite{sagemath}.

\begin{center}
\begin{tabularx}{\textwidth}{cXX} \toprule
$(\mu_1, \ldots, \mu_n)$ & $\mu_1 \cdots \mu_n \, D_{0,n}^t(\mu_1, \ldots, \mu_n)$ & $\mu_1 \cdots \mu_n \, D_{1,n}^t(\mu_1, \ldots, \mu_n)$ \\ \midrule
$(1)$ & $t$ & $0$ \\
$(2)$ & $t^2 + t$ & $0$ \\
$(3)$ & $t^3 + 3t^2 + t$ & $t$ \\
$(4)$ & $t^4 + 6t^3 + 6t^2 + t$ & $5t^2 + 5t$ \\
$(5)$ & $t^5 + 10t^4 + 20t^3 + 10t^2 + t$ & $15t^3 + 40t^2 + 15t$ \\
$(6)$ & $t^6 + 15t^5 + 50t^4 + 50t^3 + 15t^2 + t$ & $35t^4 + 175t^3 + 175t^2 + 35t$ \\ \midrule
$(1,1)$ & $t$ & $0$ \\
$(2,1)$ & $2t^2 + 2t$ & $0$ \\
$(3,1)$ & $3t^3 + 9t^2 + 3t$ & $3t$ \\
$(2,2)$ & $4t^3 + 10t^2 + 4t$ & $2t$ \\
$(4,1)$ & $4t^4 + 24t^3 + 24t^2 + 4t$ & $20t^2 + 20t$ \\
$(3,2)$ & $6t^4 + 30t^3 + 30t^2 + 6t$ & $18t^2 + 18t$ \\
$(5,1)$ & $5t^5 + 50t^4 + 100t^3 + 50t^2 + 5t$ & $75t^3 + 200t^2 + 75t$ \\
$(4,2)$ & $8t^5 + 68t^4 + 128t^3 + 68t^2 + 8t$ & $80t^3 + 200t^2 + 80t$ \\
$(3,3)$ & $9t^5 + 72t^4 + 138t^3 + 72t^2 + 9t$ & $75t^3 + 198t^2 + 75t$ \\ \midrule
$(1,1,1)$ & $2t$ & $0$ \\
$(2,1,1)$ & $6t^2 + 6t$ & $0$ \\
$(3,1,1)$ & $12t^3 + 36t^2 + 12t$ & $12t$ \\
$(2,2,1)$ & $16t^3 + 40t^2 + 16t$ & $8t$ \\
$(4,1,1)$ & $20t^4 + 120t^3 + 120t^2 + 20t$ & $100t^2 + 100t$ \\
$(3,2,1)$ & $30t^4 + 150t^3 + 150t^2 + 30t$ & $90t^2 + 90t$ \\
$(2,2,2)$ & $40t^4 + 176t^3 + 176t^2 + 40t$ & $80t^2 + 80t$ \\ \bottomrule
\end{tabularx}
\end{center}

\subsubsection*{Weingarten functions}

The value of $\SW(\sigma)$ can be calculated by inverting the orthogonality relations or by the character formula. The value of $\UW(\sigma)$ is the leading coefficient of $\SW(\sigma)$, considered as a polynomial in $M$.

\begin{center}
\begin{tabularx}{\textwidth}{cp{40mm}X} \toprule
$\sigma$ & $\UW(\sigma)$ & $\SW(\sigma)$ \\ \midrule
$(\,)$ & $1$ & $1$ \\ \midrule
$(1)$ & $\frac{1}{N}$ & $\frac{M}{N}$ \\ \midrule
$(1)(2)$ & $\frac{1}{N^2-1}$ & $\frac{M(MN-1)}{N (N^2-1)}$ \\
$(12)$ & $\frac{-1}{N(N^2-1)}$ & $\frac{-M(M - N)}{N (N^2-1)}$ \\ \midrule
$(1)(2)(3)$ & $\frac{N^2-2}{N (N^2-1) (N^2-4)}$ & $\frac{M(M^2N^2 - 2M^2 - 3MN + 4)}{N (N^2-1) (N^2-4)}$ \\
$(12)(3)$ & $\frac{-1}{(N^2-1) (N^2-4)}$ & $\frac{-M (M-N) (MN-2)}{N (N^2-1) (N^2-4)}$ \\
$(123)$ & $\frac{2}{N (N^2-1) (N^2-4)}$ & $\frac{M (M-N) (2M-N)}{N (N^2-1) (N^2-4)}$ \\ \midrule
$(1)(2)(3)(4)$ & $\frac{N^4 - 8N^2 + 6}{N^2 (N^2-1) (N^2-4) (N^2-9)}$ & $\frac{M(M^3N^4 - 8M^3N^2 + 6M^3 - 6M^2N^3 + 24M^2N + 19MN^2 - 6M - 30N)}{N^2 (N^2-1) (N^2-4) (N^2-9)}$ \\
$(12)(3)(4)$ & $\frac{-1}{N (N^2-1) (N^2-9)}$ & $\frac{-M(M - N)(M^2N^2 - 4M^2 - 5MN + 10)}{N (N^2-1) (N^2-4) (N^2-9)}$ \\
$(12)(34)$ & $\frac{N^2 + 6}{N^2 (N^2-1) (N^2-4) (N^2-9)}$ & $\frac{M(M - N)(M^2N^2 + 6M^2 - MN^3 - 6MN + 4N^2 - 6)}{N^2 (N^2-1) (N^2-4) (N^2-9)}$ \\
$(123)(4)$ & $\frac{2N^2 - 3}{N^2 (N^2-1) (N^2-4) (N^2-9)}$ & $\frac{M(M - N)(2M^2N^2 - 3M^2- MN^3 - 6MN + 3N^2 + 3)}{N^2 (N^2-1) (N^2-4) (N^2-9)}$ \\
$(1234)$ & $\frac{-5}{N (N^2-1) (N^2-4) (N^2-9)}$ & $\frac{-M(M - N)(5M^2 - 5MN + N^2 + 1)}{N (N^2-1) (N^2-4) (N^2-9)}$ \\ \bottomrule
\end{tabularx}
\end{center}

\bibliographystyle{plain}
\bibliography{deformed-monotone-hurwitz}

\end{document}